\documentclass[aap,preprint]{imsart}

\usepackage{mathtools,latexsym, enumerate, amsmath, amsthm, amssymb,hyperref, pdftricks,tikz,framed,color,enumitem,enumerate}

\pdfoutput=1

\startlocaldefs

\def\Y_#1{y_{\!#1}}

\def\E{\mathbb{E}}

\def\TV{\mathrm{TV}}
\def\cov{\operatorname{cov}}
\def\Var{\operatorname{Var}}

\def\Exp{\operatorname{Exp}}

\def\RHMC{\operatorname{RHMC}}
\def\HMC{\operatorname{HMC}}
\def\IAC{\operatorname{IAC}}
\def\MSD{\operatorname{MSD}}
\def\PiBG{\Pi_{\operatorname{BG}}}
\endlocaldefs

\newtheorem{theorem}{Theorem}[section]

\newtheorem{lemma}[theorem]{Lemma}

\newtheorem{remark}[theorem]{Remark}
\newtheorem{prop}[theorem]{Proposition}

\newtheorem{hypothesis}{Hypothesis}[section]

\theoremstyle{definition}
\newtheorem{defn}{Definition}[section]
\newtheorem{algorithm}{Algorithm}[section]

\theoremstyle{remark}
\newtheorem{rem}{Remark}[section]

\begin{document}

\begin{frontmatter}

\title{Randomized Hamiltonian Monte Carlo}
\runtitle{Randomized HMC}


\begin{aug}
\author{\fnms{Nawaf} \snm{Bou-Rabee}\thanksref{t1,m1}\corref{}\ead[label=e1]{nawaf.bourabee@rutgers.edu}}
\and
\author{\fnms{Jes\'{u}s} \snm{Mar\'{i}a Sanz-Serna}\thanksref{t2,m2}\ead[label=e2]{jmsanzserna@gmail.com}}

\affiliation{Rutgers University Camden\thanksmark{m1} and  Universidad Carlos III de Madrid\thanksmark{m2}}

\address{Department of Mathematical Sciences \\ Rutgers University Camden \\ 311 N 5th Street \\ Camden, NJ 08102 \\ \printead{e1}}
\address{Departamento de  Matem\'aticas \\ Universidad Carlos III de Madrid \\ Avenida de la Universidad 30 \\ E--28911 Legan\'es (Madrid), Spain \\ \printead{e2}}

\thankstext{t1}{N.~Bou-Rabee was partly supported by NSF grant DMS-1212058.}
\thankstext{t2}{J.~M.~Sanz-Serna has been supported by  Ministerio de Eco\-nom\'{\i}a y Comercio, Spain (proj\-ect  MTM2013-46553-C3-1-P).}
\end{aug}

\runauthor{N. Bou-Rabee \and J.~M.~Sanz-Serna}

\begin{abstract}
Tuning the durations of the Hamiltonian flow in Hamiltonian Monte Carlo (also called Hybrid Monte Carlo) (HMC) involves a tradeoff between computational cost and sampling quality, which is typically challenging to resolve in a satisfactory way.   In this article we present and analyze a randomized HMC method (RHMC), in which these durations are i.i.d.~exponential random variables whose mean is a free parameter.   We focus on the small time step size limit, where the algorithm is rejection-free and the computational cost is proportional to the mean duration.   In this limit, we prove that RHMC is geometrically ergodic under the same conditions that imply geometric ergodicity of the solution to underdamped Langevin equations.  Moreover, in the context of a multi-dimensional Gaussian distribution, we prove that the sampling efficiency of RHMC, unlike that of constant duration HMC, behaves in a regular way. This regularity is also verified numerically in non-Gaussian target distributions. Finally we suggest variants of RHMC for which the time step size is not required to be small.
\end{abstract}




\end{frontmatter}


\section{Introduction}

In the present article we suggest  a randomized version of the Hamiltonian Mon\-te Carlo (also called Hybrid Monte Carlo) algorithm that, under very general hypotheses, may be proved to be {\em geometrically ergodic}.
The Hamiltonian Monte Carlo (HMC) algorithm is a general purpose Markov Chain Monte Carlo (MCMC) tool for sampling from a probability distribution $\Pi$ \cite{DuKePeRo1987,Li2008,Ne2011,Sa2014}. It offers the potential of generating proposed moves that are far away from the current location of the chain and yet may be accepted with high probability. The algorithm is based on integrating a Hamiltonian system and possesses two free parameters: the duration of the Hamiltonian flow and the  time step size of the integrator.
Unfortunately, the performance of HMC depends crucially on the values assigned by the user to those parameters; while for some parameter values HMC may be highly efficient, it is well known that, as discussed below, there are values for which the algorithm, in its simplest form, is not even ergodic.
The Randomized Hybrid Monte Carlo (RHMC) addresses these shortcomings of HMC.

We target probability distributions of the form
\begin{equation} \label{eq:target}
\Pi(dq) = C_0^{-1} \exp(- \Phi(q)) dq \;, \quad C_0 = \int_{\mathbb{R}^D} \exp(- \Phi(q)) dq\;,
\end{equation}
where the negative loglikelihood $\Phi: \mathbb{R}^D \to \mathbb{R}$ is seen as a potential energy function.  We assume that $\Phi$ is at least differentiable and such that $C_0<\infty$. As is the case with other MCMC methods, HMC does not require knowing the normalization factor $C_0$.  HMC enlarges the state space from $\mathbb{R}^D$ to $\mathbb{R}^{2D}$ and considers the Boltzmann-Gibbs distribution in this space
\begin{equation} \label{eq:boltzmann_gibbs}
 \PiBG(dq,dp) = C_0^{-1}(2 \pi)^{-D/2}  \exp(- H(q,p) ) dq dp  \;,
\end{equation}
where the artificial Hamiltonian function (or total energy) $H: \mathbb{R}^{2D} \to \mathbb{R}$ is taken to be
\begin{equation} \label{eq:hamiltonian}
H(q,p) = \frac{|p|^2}{2} + \Phi(q)\;.
\end{equation}
The vector $p$ plays the role of a mechanical momentum and the term $|p|^2/2$ is the corresponding kinetic energy. The target $\Pi$ is the marginal of $\PiBG$ on $q$; the marginal on $p$ is, of course,
the standard $D$-dimensional Gaussian $\mathcal{N}(0,1)^D$. More complicated kinetic energies of the form
$(1/2) p^TM^{-1}p$, with $M$ a symmetric positive definite mass matrix, may also be used, but, for notational simplicity,
we restrict our study to the Hamiltonian \eqref{eq:hamiltonian}.

The basic idea of HMC is encapsulated in the following procedure, where the duration $\lambda>0$ is a (deterministic) parameter whose value is specified by the user.
\begin{algorithm}[HMC] \label{algo:hmc}
Given the duration parameter $\lambda>0$ and the current state of the chain $X_0\in\mathbb{R}^D$, the method outputs a state $X_1\in\mathbb{R}^D$ as follows.
\begin{description}
\item[Step 1] Generate a $D$-dimensional random vector  $\xi_0 \sim \mathcal{N}(0,1)^D$.
\item[Step 2] Evolve over $[0,\lambda]$ Hamilton's equations associated to \eqref{eq:hamiltonian}
\begin{equation} \label{eq:hamiltonian_dynamics}
\begin{aligned}
& \dot{q} = p, \\
& \dot{p} = - \nabla \Phi(q),
\end{aligned}
\end{equation}
 with initial condition $(q(0), p(0)) = (X_0,\xi_0)$.
\item[Step 3] Output $X_1 = q(\lambda)$.
\end{description}
\end{algorithm}

Since Step 2 conserves the Boltzmann-Gibbs distribution, it is clear that the mapping $X_0\mapsto X_1$ preserves the target $\Pi$ and therefore may be used to generate a Markov Chain having $\Pi$ as an invariant distribution
(in fact the resulting chain is actually reversible with respect to $\Pi$). Note that Step 1 is easy to perform since it only involves generating a $D$-dimensional normal random vector. This step  is the only source of randomness in determining $X_1$ conditional on $X_0$.
The Hamiltonian flow in Step 2 is what, in principle, enables HMC to make large moves in state space that reduce correlations in the Markov chain $\{ X_i \}$.
Roughly speaking, one may hope that,  by increasing $\lambda$, $X_1$ moves away from $X_0$, thus reducing  correlation.  However, simple examples show that this outcome is far from assured. Indeed, for the univariate normal distribution with $\Phi(q) = q^2/2$,  Hamilton's equations coincide with those of the harmonic oscillator, $(d/dt)q = -p$, $(d/dt)q = p$, and the flow is a rotation in the $(q,p)$ plane with period $2\pi$. It is easy to see (see Section 5 for a fuller discussion) that, if $X_0$ is taken from the target distribution, as $\lambda$ increases from $0$ to $\pi/2$, the correlation between $X_1$ and $X_0$ decreases and for $\lambda = \pi/2$, $X_1$ and $X_0$ are independent. However increasing $\lambda$ beyond $\pi/2$ will cause an increase of the correlation and for $\lambda = \pi$, $X_1=-X_0$ and the chain is not ergodic.  For general distributions, it is likely that a small $\lambda$ will lead to a highly correlated chain, while choosing $\lambda$ too large  may cause the Hamiltonian trajectory to make a U-turn and fold back on itself, thus increasing correlation \cite{HoGe2014}.

 In practice, a formula for the exact solution used in Step 2 is  unavailable and a numerical solution is used instead. Thus, in addition to the duration $\lambda$ of the Hamiltonian flow in Step 2, another key parameter in the HMC method is the time step size $\Delta t$ used to generate this numerical solution. To correct the  bias introduced by time discretization error, a Metropolis-Hastings accept-reject step is also added \cite{MeRoRoTeTe1953,Ha1970}.    In order to keep the Metropolis-Hastings ratio  simple, typically a volume-preserving and reversible method is used to numerically simulate Hamilton's equations in Step 2 \cite{FaSaSk2014}. The integrator of choice is the Verlet method, which is second-order accurate and, like Euler's rule, only requires one new  evaluation of the gradient $\nabla \Phi(q)$ per step. Unfortunately, time discretization does not remove the complex dependence of correlation on the duration parameter $\lambda$. For instance, in the preceding  example where $\Phi(q) = q^2/2$, it is easy to check that if $\lambda$ is close to an integer multiple of $\pi$ and $\Delta t>0$ is suitably chosen, the Verlet numerical integration will result, for each $X_0$, in $X_1 = -X_0$ (a move that will be accepted by the Metropolis-Hasting step). To avoid such poor performance, HMC is typically operated with values of $\Delta t$ that are randomized \cite{Ma1989,Ne2011}. Since, due to stability restrictions, explicit integrators cannot operate with arbitrarily large values of the time step, $\Delta t$ is typically chosen  from a uniform distribution in an (often narrow) interval $(\Delta t_{\rm min}, \Delta t_{\rm max})$. In any case, the fact remains that increasing the duration parameter will increase the computational cost and may impair the quality of the sampling.

In this paper we randomize the duration of the Hamiltonian flow, cf.~\cite{Ma1989,CaLeSt2007,Ne2011}. More precisely, in RHMC the lengths of the time intervals of integration of the Hamiltonian dynamics at the different steps of the Markov chain are identically distributed exponential random variables; these durations are mutually independent and independent of the state of the chain.  In what follows we are primarily interested in the case where the procedure uses the exact Hamiltonian flow (or, in practical terms, where the integration is carried out with such a small value of $\Delta t$ which ensures that essentially all steps of the chain result in acceptance).  This leaves the mean duration as the only free parameter.  In this exact integration scenario, we prove that, regardless of the choice of the mean duration, RHMC is geometrically ergodic.  Furthermore, we show that the dependence of the performance of RHMC on this mean duration parameter is simpler than the dependence of the performance of HMC on its constant duration parameter.   A full discussion of the situation where time-discretization errors are taken into account requires heavy use of numerical analysis techniques and will be the subject of a future publication.  Nevertheless in Section~\ref{sec:outlook} we present two variants of RHMC, based on the ideas of Markov Chain Approximation methods (MCAM) \cite{KuDu2001, BoVa2016}, that replace the Hamiltonian flow by a numerical approximation and may be treated with the infinitesimal tools we employ in the exact integration scenario.

Section 2  provides a description of the RHMC method and its infinitesimal generator $L$.  In Section 3 we prove that the measure
$ \PiBG$ is  infinitesimally invariant  for $L$.  We then construct a Lyapunov function for RHMC that is of the same form as that used for the  Langevin equations and requires similar assumptions on the potential energy function \cite{SaSt1999, MaStHi2002, Ta2002}. Here it is important to point out that, while the Langevin dynamics explicitly includes friction, the dissipative behavior of RHMC comes from the randomization of the momentum. In particular, if the chain is at a location of high potential energy, the Hamiltonian dynamics will typically change a large part of the potential energy into kinetic energy; then, with high probability, the next momentum randomization will decrease the kinetic energy.
With this Lyapunov function, we extend a local semimartingale representation of the process.  It also follows that Dynkin's formula holds for functions that satisfy a mild growth condition.
Using Dynkin's formula we prove that the measure $ \PiBG$ is invariant  for RHMC.  Using a generating function for the Hamiltonian flow and a Duhamel formula for the Markov semigroup associated with the RHMC process, we prove that the transition probability distribution of RHMC satisfies a minorization condition.  We then invoke Harris' theorem to conclude that RHMC is geometrically ergodic with respect to $\PiBG$, for any choice of the mean duration parameter $\lambda$.

Section 4 considers the model problem of a multi-dimensional Gaussian target distribution \cite{Ma1989,Ne2011}.  For both RHMC and HMC, explicit formulas are derived for (i) the integrated autocorrelation time associated to the standard estimator of the mean of the target distribution; and (ii) the equilibrium mean-squared displacement.   These formulas imply that the sampling efficiency of RHMC behaves in a simple way, while the sampling efficiency of HMC displays complicated dependence on the duration parameter.
Section 5 presents numerical tests for a two-dimensional double well potential energy and a potential energy used in the chemistry literature for a pentane molecule.  These tests support our theoretical findings.  Two variants of RHMC that do not assume that integration errors are negligible are suggested in Section 6. The first of them is easily proved to have a Lyapunov function under suitable hypotheses but introduces a bias in the target distribution; the second removes the bias by allowing momentum flips at a rate dictated by a Metropolis-Hastings ratio. There is an Appendix devoted to Harris theorem.

To summarize, the main theoretical contributions of this paper are the following.
\begin{itemize}
\item a proof of a Foster-Lyapunov drift condition for the infinitesimal generator of RHMC;
\item a solution to a martingale problem for RHMC;
\item a proof that infinitesimal invariance of the measure $\PiBG$ and Dynkin's formula imply that the measure $\PiBG$ is invariant for RHMC;
\item a minorization condition for RHMC;
\item a proof that $\PiBG$ is the unique invariant measure of RHMC and that RHMC is geometrically ergodic (which combines all of the previous results); and,
\item introduced two practical implementations of RHMC, including one that is unbiased.
\end{itemize}

Let us finish this introduction by pointing out that the RHMC method is related to Anderson's impulsive thermostat common in molecular dynamics, which describes a molecular system interacting with a heat bath \cite{An1980, Li2007,ELi2008}.  The molecular system is modeled using Hamiltonian dynamics, and its interaction with a heat bath is modeled by collisions that cause an instantaneous randomization of the momentum of a randomly chosen particle.  The times between successive collisions are assumed to be i.i.d.~exponential random variables.  In Ref.~\cite{ELi2008},  E and Li prove that the Anderson thermostat on a hyper-torus is geometrically ergodic.  Since the state space is bounded, the main issue in that proof is the derivation of a minorization condition for the process.  Our proof of geometric ergodicity of the RHMC method can be modified to extend their result to an unbounded space.


\section{RHMC Method} \label{sec:rhmc}

Here we provide step-by-step instructions to produce an RHMC trajectory, and afterwards, introduce the infinitesimal generator $L$ of RHMC.


The RHMC method generates a right-continuous with left limits (c\`adl\`ag) Markov process $Z_t$.  While Algorithm~\ref{algo:hmc} was formulated in $\mathbb{R}^{D}$, the process $Z_t = (Q_t, P_t)$ is defined in the enlarged space $\mathbb{R}^{2 D}$ to include the possibility of partial randomization of the momentum, as in the generalized Hybrid Monte Carlo of Horowitz \cite{Ho1991, KePe2001, AkRe2008}.  This process $Z_t$ can be simulated by iterating the following steps. The mean duration $\lambda>0$ and the Horowitz angle $\phi \in (0, \pi/2]$ are deterministic parameters.
\begin{algorithm}[RHMC] \label{algo:rhmc}
Given the current time $t_0$ and the current state $Z_{t_0} = (Q_{t_0}, P_{t_0})$, the method computes the next momentum randomization time $t_1>t_0$ and the path of the process $Z_{s} = (Q_{s},P_{s})$ over $(t_0, t_1]$ as follows.
\begin{description}
\item[Step 1] Update time via $t_1 = t_0 + \delta t_0$ where $ \delta t_0 \sim \Exp(1/\lambda)$.
\item[Step 2] Evolve over $[t_0, t_1]$  Hamilton's equations \ref{eq:hamiltonian_dynamics}) associated with \eqref{eq:hamiltonian}
with initial condition $(q(t_0), p(t_0))  = (Q_{t_0}, P_{t_0})$.
\item[Step 3] Set \[
Z_s = (Q_s, P_s) = (q(s), p(s) ) \quad \text{for $t_0 \le s < t_1$} \;.
\]
\item[Step 4] Randomize momentum by setting \[
Z_{t_1} = (q(t_1), \cos(\phi) p(t_1) + \sin(\phi) \xi )
\] where $\xi \sim \mathcal{N}(0,1)^D$.
\item[Step 5] Output $X_1= q(t_1)$.
\end{description}
\end{algorithm}


On test functions $f \in C^1(\mathbb{R}^{2D})$, the infinitesimal generator of $Z$ is given by
 \begin{equation} \label{eq:generator}
L f(q,p) = \overbrace{\lambda^{-1} \E \left\{ f(  \Gamma (q,p) ) - f(q,p) \right\}}^\text{momentum randomization operator} + \underbrace{p^T \nabla_q f(q,p) - \nabla_q \Phi(q)^T \nabla_p f(q,p)}_\text{Liouville operator},
\end{equation}
The expectation in the momentum randomization operator is over the random variable $\Gamma(q,p)$ defined as
\begin{equation} \label{eq:momentum_randomization}
\Gamma(q,p) = (q, \cos(\phi) p + \sin(\phi) \xi)\;,
\end{equation}
where $\xi \sim \mathcal{N}(0,1)^D$ and $\phi \in (0, \pi/2]$.  A sample path of this process is given by the Hamiltonian flow associated with  \eqref{eq:hamiltonian} with intermittent and instantaneous jumps in momentum.  The random times between successive jumps are independent and exponentially distributed with mean $\lambda$.
 The Horowitz angle $\phi$ is a deterministic parameter that governs how much the  momentum immediately after a jump depends on the value of the momentum immediately prior to a jump.  The case $\phi=0$ in \eqref{eq:momentum_randomization} has to be excluded because it leads to $\Gamma(q,p)=(q,p)$ and then the generator $L$ reduces to the Liouville operator associated with the Hamiltonian $H$, which is in general not ergodic with respect to $\PiBG$.  Note that, if $\phi=\pi/2$, the random vector $\Gamma(q,p)$ does not depend on $p$ (complete momentum randomization).





\section{Geometric Ergodicity of RHMC} \label{sec:non_asymptotic}

\subsection{Overview}

In this section we prove that $\PiBG$ is the unique invariant probability measure of RHMC and that RHMC is geometrically ergodic.  Our main tool to prove geometric ergodicity is Harris Theorem.  In Appendix~\ref{sec:harris_theorem}, we recall this theorem in the present context of a continuous-time Markov process with an uncountable state space.  

A main ingredient in Harris theorem is Hypothesis~\ref{driftcondition} on the existence of a Lyapunov function, which we refer to as a {\em Foster-Lyapunov drift condition}. We formulate this condition in terms of an abstract infinitesimal generator whose precise meaning is given in Definition~\ref{defn:generator}.   Note that this definition of an infinitesimal generator accommodates the Lyapunov function $V: \mathbb{R}^{2D} \to \mathbb{R}$ introduced below because for any $t>0$ the process \[
V(Z_t) - V(z) - \int_0^t L V(Z_s) ds \;, \quad Z_0 = z \in \mathbb{R}^{2D}\;,
\]  is always a local martingale, and hence, according to Definition~\ref{defn:generator} the function $V$ belongs to the domain of $L$.  Note that here we assume that this Lyapunov function is continuously differentiable, since the operator $L$ involves partial derivatives.   After showing that $V$ is a Lyapunov function for $L$, we apply this Lyapunov function to solve the martingale problem for the operator $L$ on functions that are $C^1$ with globally Lipschitz continuous derivative.  This solution is used to show that $\PiBG$ is an invariant measure for RHMC.  We then prove that the transition probabilities of RHMC satisfy a minorization condition given in Hypothesis~\ref{minorization}.   With these pieces in place, we invoke Harris Theorem to conclude that $\PiBG$ is the unique invariant measure for RHMC, and that RHMC is geometrically ergodic.

To establish a Foster-Lyapunov drift condition, we use a Lyapunov function $V$ originally introduced to prove stochastic stability of a Hamiltonian system with dissipation and random impulses; see \S2.2 and equation (14) of \cite{SaSt1999}.  See also Section 3 of Ref.~\cite{MaStHi2002} and Section 2 of Ref.~\cite{Ta2002} for an application of this Lyapunov function to prove geometric ergodicity of the solution to underdamped Langevin equations with non-globally Lipschitz coefficients.   This Lyapunov function $V: \mathbb{R}^{2D} \to \mathbb{R}$ is of the form:
\begin{equation} \label{eq:lyapunovfunction}
V(z) = H(z) + c_1 \langle q, p \rangle + c_2 \frac{|q|^2}{2} \;, \quad z=(q,p) \in \mathbb{R}^{2D} \;,
\end{equation}
where $H$ is the Hamiltonian given earlier in \eqref{eq:hamiltonian}, and $c_1$ and $c_2$ are constants given by:
\begin{equation} \label{eq:c1+c2}
c_1 = \frac{\lambda^{-1}}{4} \sin^2(\phi) ~~\text{and}~~ c_2 =  \lambda^{-1} c_1 (1 -  \cos(\phi) ) \;.
\end{equation}
Since $\lambda>0$ and $\phi \in (0, \pi/2]$, note from \eqref{eq:c1+c2} that both $c_1$ and $c_2$ are positive.

Throughout this section we will use the following conditions on the potential energy $\Phi(q)$.  Note that not all of these assumptions will be required for every statement, but we find it notationally convenient to have a single set of assumptions to refer to.
\begin{hypothesis} \label{hypothesis_on_phi}
The potential energy $\Phi \in C^2( \mathbb{R}^D)$ satisfies the following conditions.
\begin{description}
\item[H1.]  One has $\Phi(q) \ge 0$ and $\int_{\mathbb{R}^{D}} |q|^2 \exp(-\Phi(q)) dq < \infty$.
\item[H2.] Let $c_1$ and $c_2$ be the constants appearing in \eqref{eq:lyapunovfunction}. Then there exist $a >0$ and $ b\in (0,1)$ such that \[
\frac{1}{2} \langle \nabla_q \Phi(q), q \rangle \ge b \Phi(q) + \frac{ (c_1 b)^2 + c_2 b (1-b) }{2 (1-b)} |q|^2 - a
\] for all $q \in \mathbb{R}^D$. 
\end{description}
\end{hypothesis}
We stress that these assumptions on the potential energy function are typically made to prove geometric ergodicity of the solution to underdamped Langevin equations. For instance, see Equation (13) of Ref.~\cite{SaSt1999}, Hypothesis 1.1 in Ref.~\cite{Ta2002}, and Condition 3.1 in Ref.~\cite{MaStHi2002}.   Hypothesis {\bf H1} and \eqref{eq:lyapunovfunction} imply that for any constants $c_1, c_2 \in \mathbb{R}$, we have that \[
\int_{\mathbb{R}^{2D}} V(z) \PiBG(dz) < \infty \;.
\]
In other words, this Lyapunov function is integrable with respect to $\PiBG$.   

The hypothesis that the potential is bounded from below by itself guarantees that the kinetic energy and therefore the momenta are bounded as time increases. It follows that the configuration variables grow at most linearly with time and therefore solutions of Hamilton's equations are well-defined for all real time.   

%

\subsection{The Measure $\PiBG$ is an Infinitesimally Invariant Measure}

As expected, the Boltzmann-Gibbs distribution in \eqref{eq:boltzmann_gibbs} is an infinitesimally invariant probability measure for the process $Z$. By implication the target $\Pi$ is infinitesimally invariant for the process $Q$.   To state this result, let $C^{\infty}_c(\mathbb{R}^{2D})$ denote the space of compactly supported smooth functions from $\mathbb{R}^{2D}$ to $\mathbb{R}$.

%
%

\begin{prop} \label{prop:infinitesimal_invariance}
Suppose Hypothesis~\ref{hypothesis_on_phi} {\bf H1} holds.  Then for any $f \in C^{\infty}_c(\mathbb{R}^{2D})$ we have that \[
\int_{\mathbb{R}^{2D}} L f(z) \; \PiBG(dz) = 0 \;.
\]
\end{prop}

\begin{proof}
Hypothesis~\ref{hypothesis_on_phi} {\bf H1} guarantees that $\int_{\mathbb{R}^{2D}} e^{-H(z)} dz < \infty$, and hence, the measure $\PiBG$ is a probability measure.
Note that $\PiBG$ is an invariant probability measure for the momentum randomization operator in \eqref{eq:generator}, and hence, \[
\int_{\mathbb{R}^{2D}} \E \left\{ f(\Gamma(q,p)) - f(q,p) \right\}  \PiBG(dq,dp) = 0 \;.
\]
Moreover, integration by parts shows that the Liouville operator leaves $\PiBG$ infinitesimally invariant.  In particular, the boundary terms resulting from the integration by parts vanish  because $f$ is compactly supported.  
\end{proof}

Later in this section, we strengthen this result to $\PiBG$ is the unique invariant probability measure for RHMC.

\subsection{Foster-Lyapunov Drift Condition}

The following Lemma is remarkable because it states that the infinitesimal generator $L$ possesses a Lyapunov function, even though RHMC does not incorporate explicit dissipation.  

%
%

\begin{lemma} \label{lemma:rhmc_drift_condition}
Suppose Hypothesis~\ref{hypothesis_on_phi} holds, and $V(z)$ is given by \eqref{eq:lyapunovfunction}.  Then 
there exist real positive constants $c_1$, $c_2$, $\gamma$ and $K$ such that: \begin{equation}
L V(z) \le - \gamma V(z) + K
\end{equation}
 for all $z \in \mathbb{R}^{2D}$. Moreover, $V$ is nonnegative and  \[
 \lim_{|z| \to \infty} V(z) = \infty \;.
\]
\end{lemma}

This Lemma implies that Hypothesis~\ref{driftcondition} of Harris Theorem holds.  The proof below shows that the momentum randomizations in RHMC are the source of dissipation in RHMC.

\begin{proof}
Let $z=(q,p) \in \mathbb{R}^{2D}$.  From \eqref{eq:lyapunovfunction}, note that:
\begin{eqnarray*}
\nabla_q V &=& \nabla_q \Phi + c_1 p + c_2 q \\
\nabla_p V &=& p + c_1 q
\end{eqnarray*}
and if we set $u = \cos(\phi) p + \sin(\phi) \xi$, note from \eqref{eq:momentum_randomization} that:
\begin{eqnarray*}
\E |u|^2 &=& D  \sin^2(\phi)  + \cos^2(\phi) |p|^2  \\
\E \langle q, u \rangle &=& \cos(\phi) \langle q, p \rangle
\end{eqnarray*}
where the expected value is taken over the $D$-dimensional standard normal vector $\xi \sim \mathcal{N}(0,1)^D$.
We recall that we excluded the case $\phi=0$ in the definition of $u$ in \eqref{eq:momentum_randomization}.  A direct calculation shows that: \begin{equation} \label{eq:tLV}
\begin{aligned}
& L V =  \lambda^{-1} D  \frac{1}{2} \sin^2(\phi)  - c_1 \langle \nabla_q \Phi, q \rangle    \\
& \quad +  (2 c_1 - \lambda^{-1} \sin^2(\phi) )  \frac{|p|^2}{2}
+ ( \lambda^{-1} c_1 (\cos(\phi) - 1) + c_2 ) \langle q, p \rangle  \;.
\end{aligned}
\end{equation}
We now choose $c_1$ and $c_2$ such that \begin{equation*}
\begin{cases}
  2 c_1 - \lambda^{-1} \sin^2(\phi)  = - 2 c_1 \\
 \lambda^{-1} c_1 (\cos(\phi) - 1) + c_2  = 0 
  \end{cases}
\end{equation*} In other words, we pick $c_1$ and $c_2$ to eliminate the $ \langle q, p \rangle$ cross term in $LV$, and to rescale the $(|p|^2)/2$ term so that it matches the coefficient of the $\langle \nabla_q \Phi, q \rangle$ term.  Solving these equations yields \eqref{eq:c1+c2}. 
With this choice of $c_1$ and $c_2$, \eqref{eq:tLV} simplifies to: \begin{equation}  \label{eq:LV}
L V = -2 c_1 \left( \frac{|p|^2}{2} + \frac{1}{2} \langle \nabla_q \Phi, q \rangle \right)  +  D \lambda^{-1}  \frac{1}{2}  \sin^2(\phi)  \;.
\end{equation}

Let $b$ be the constant appearing in Hypothesis~\ref{hypothesis_on_phi} {\bf H2}.   Applying Cauchy's inequality with $\delta>0$ to the $ \langle q, p \rangle$ cross term in \eqref{eq:lyapunovfunction} yields: \[
b V(q,p) \le \frac{b}{2} |p|^2 + b \Phi(q) + c_1 b \left( \frac{\delta}{2} |p|^2 + \frac{1}{2 \delta} |q|^2 \right) + \frac{c_2 b}{2}  |q|^2 \;.
\]  Choosing $\delta = (1-b)/(c_1 b)$ and invoking Hypothesis~\ref{hypothesis_on_phi} {\bf H2}, we obtain\begin{align}
b V(q,p) &\le \frac{|p|^2}{2} + b \Phi(q) + \frac{ (c_1 b)^2 + c_2 b (1-b)}{2 (1-b)} |q|^2 \nonumber \\
&\le \frac{|p|^2}{2} + \frac{1}{2} \langle \nabla \Phi(q), q \rangle + a  \label{eq:v2} \;.
\end{align}
Together \eqref{eq:LV} and \eqref{eq:v2} imply that the desired Foster-Lyapunov drift condition holds with $\gamma = 2 c_1 b$ and for some $K>0$.

To finish the proof, recall that $\Phi \ge 0$ by Hypothesis~\ref{hypothesis_on_phi} {\bf H1}, and thus, it suffices to check that the quadratic form \[
 \frac{|p|^2}{2} + c_1 \langle q, p \rangle + c_2 \frac{|q|^2}{2} 
\] appearing in $V(q,p)$ is positive definite. This condition is met when  \[
 c_2 > c_1^2 >0 \implies (1- \cos(\phi) ) > \frac{1}{4} \sin^2(\phi) >0
\] which holds for all $\phi \in (0, \pi/2]$.
\end{proof}

\begin{remark}
The proof of Lemma~\ref{lemma:rhmc_drift_condition} shows that $\gamma \propto \lambda^{-1} \sin^2(\phi)$.  Thus, we see that if $\lambda$ is smaller or $\phi$ is closer to $\pi/2$, then $\gamma$ becomes larger.  This result is expected since momentum randomizations are the source of dissipation in RHMC, and smaller $\lambda$ implies more randomizations of momentum, and larger $\phi$ implies the momentum is randomized more completely.
\end{remark}

\begin{remark}
It follows from Lemma~\ref{lemma:rhmc_drift_condition} that: \[
\E_z V(Z_t) \le e^{-\gamma t} V(z) + \frac{K}{\gamma} (1-e^{-\gamma t} ) \;.
\]
See, e.g., the proof of Theorem 6.1 in Ref.~\cite{MeTw1993}.
\end{remark}


\subsection{Martingale Problem}

Here we use Lemma~\ref{lemma:rhmc_drift_condition} to solve the martingale problem for the operator $L$ on functions that are $C^1$ with globally Lipschitz continuous derivative.
\begin{lemma} \label{lemma:martingale_problem_solution}
Suppose Hypothesis~\ref{hypothesis_on_phi} holds.   For all globally Lipschitz and continuously differentiable functions $f: \mathbb{R}^{2D} \to \mathbb{R}$, for any initial condition $Z_0 = z \in \mathbb{R}^{2D}$, and for any $t>0$, the local martingale \[
M^{f}_t = f(Z_t) - f(z) -\int_0^t L f(Z_s ) ds 
\] is a martingale.
\end{lemma}
A key ingredient in the proof given below is the Lyapunov function for $L$ from Lemma~\ref{lemma:rhmc_drift_condition}.  In particular, since globally Lipschitz functions grow at most linearly, the class of functions that appear in Lemma~\ref{lemma:martingale_problem_solution} are bounded by this Lyapunov function.  Moreover, Dynkin's formula holds for this class of functions: \begin{equation} \label{eq:dynkins_formula}
\E_z f(Z_t) = f(z) +\int_0^t \E_z L f (Z_s) ds  \;, \quad Z_0 = z\;,  \quad t \ge 0\;.
\end{equation}
See, e.g., Chapter 1 of Ref.~\cite{Da1993} for more discussion on Dynkin's formula.

\begin{proof}
In this proof, we use the well-known fact (see, e.g., Corollary 3, Chapter II in Ref.~\cite{protter2004stochastic})  that a local martingale with integrable quadratic variation is a martingale.
Let $[ M^{f} ](t)$ denote the quadratic variation of $M^f_t$ on the interval $[0,t]$. 
The global Lipschitz assumption on $f$ implies that there exists a constant $L_{f}>0$ such that: \[
| f(z_1) - f(z_2) | \le L_{f} \; | z_1 - z_2 |
\] for all $z_1, z_2 \in \mathbb{R}^{2D}$.  Moreover, since the process $Z_t$ satisfies Hamilton's differential equations in between consecutive momentum randomizations, and since the process $f(Z_t)-M^f_t$ is continuous and of finite variation, the quadratic variation of $M^f$ is equal to the sum of the squares of the jumps in $f(Z_t)$.  Thus,
\begin{align*}
[ M^{f} ](t) &= [f](t) = \sum_{0< s \le t} ( f(Z_s) - f(Z_{s-}) )^2 \le \sum_{0 < s \le t} L_f^2 |Z_s - Z_{s-}|^2 \;.
\end{align*}
In other words, the global Lipschitz property of $f$ enables bounding the quadratic variation of the scalar-valued process $f(Z_t)$ by the quadratic variation of the components of the process $Z_t$.   

Let $\{ t_i \}$ be the sequence of random times at which the momentum randomizations occur.  This sequence can be produced by iterating the recurrence relation: $t_{i+1} = t_i + \delta t_i$ with initial condition $t_0=0$ and where $\{ \delta t_i \}$ are i.i.d.~exponential random variables with mean $\lambda$.   Let $N_t$ be the (random) number of momentum randomizations that have occurred up to time $t$.  Note that $N_t$ is a Poisson process with intensity $\lambda^{-1}$, and hence, a.s.~bounded.  This permits us to interchange expectation and summation in the following inequalities,
\begin{align*}
\E_z [ M^{f} ](t) &\le  L_f^2 \E_z \sum_{1 \le i \le N_t} |Z_{t_i} - Z_{t_i-}|^2 \\
&\le L_f^2   \lambda^{-1}  \sum_{i > 0}   \E_z \left\{  |P_{t_{i}} - P_{t_{i}-}|^2 (t \wedge t_{i+1} - t \wedge t_i ) \right\}  \\
&\le  2 L_f^2 \lambda^{-1}  \sum_{i > 0}  \E_z \left\{ ( | P_{t_{i}}|^2 + | P_{t_{i}-} |^2 ) (t \wedge t_{i+1} - t \wedge t_i ) \right\}  \\
&\le  4 L_f^2 C_2 \lambda^{-1}  \sum_{i > 0}  \E_z \left\{ (  V(Z_{t_{i}}) + V(Z_{t_{i}-})) (t \wedge t_{i+1} - t \wedge t_i ) \right\}  \\
&\le  8 L_f^2 C_2 \lambda^{-1}   t  \left(V(z) + \frac{K}{\gamma}  \right)  
\end{align*}
where in the last two steps we used the Lyapunov function given in Lemma~\ref{lemma:rhmc_drift_condition}, and in addition, we introduced the positive constant $C_2 = 2 c_2 / (c2-c_1^2)$ with $c_1$ and $c_2$ being the constants defined in \eqref{eq:c1+c2}.  Thus, we may conclude that $M^{f}_t$ is a martingale for any $t > 0$.

Alternatively, we could have used the compensator of $[M^{f}](t)$:\begin{align*}
\langle M^{f} \rangle(t) &= \int_0^t \Big( L f^2(Z_s) - 2 f(Z_s) L f(Z_s)  \Big) \; ds   \\
&= \lambda^{-1} \int_0^t  \E \Big( f(\Gamma(Z_s)) - f(Z_s) \Big)^2 \; ds 
\end{align*}
which would give a similar bound on $\E_z [ M^{f} ](t)$.
\end{proof}

\subsection{The Measure $\PiBG$ is an Invariant Measure}

In this part, we combine Prop.~\ref{prop:infinitesimal_invariance} and Lemma~\ref{lemma:martingale_problem_solution} to prove that $\PiBG$ is an invariant probability measure for RHMC.   
\begin{lemma}
Suppose Hypothesis~\ref{hypothesis_on_phi} holds.   For any $f \in C^{\infty}_c( \mathbb{R}^{2D})$ and for any $t>0$, \[
\int_{\mathbb{R}^{2 D}} \E_z f(Z_t) \PiBG(dz) = \int_{\mathbb{R}^{2 D}} f(z) \PiBG(dz) \;.
\]
\end{lemma}
To prove this Lemma, we use Dynkin's formula and condition on a fixed sequence of jump times to exploit the fact that the Hamiltonian flow and the momentum randomization individually leave $\PiBG$ invariant.  
\begin{proof}


Let $f \in C^{\infty}_c( \mathbb{R}^{2D})$ and $z=(q,p) \in \mathbb{R}^{2D}$.  Referring to \eqref{eq:generator}, since $f \in C^{\infty}_c( \mathbb{R}^{2D})$ the function \[
\E f(\Gamma(z)) = \E f(q, \cos(\phi) p + \sin(\phi) \xi) 
\] is smooth, compactly supported in the $q$ component, and bounded in the $p$ component, and hence, $Lf \in C^{\infty}_b(\mathbb{R}^{2D})$.  Moreover, since any smooth function with compact support is globally Lipschitz continuous, we can invoke Lemma~\ref{lemma:martingale_problem_solution} to conclude that for any $f \in C^{\infty}_c(\mathbb{R}^{2D})$ and for any $t>0$, the process \[
f(Z_t) - f(z) - \int_0^t L f(Z_s) ds  \;, \quad Z_0 = z \;,
\] is a martingale.  Thus, Dynkin's formula holds \[
\E_z f(Z_t) = f(z) + \int_0^t \E_z Lf(Z_s) ds \;,
\] and in particular, \begin{equation}
\begin{aligned} \label{eq:averaged_dynkin}
  \int_{\mathbb{R}^{2D}}   \E_z f(Z_t) \PiBG(dz) &= \int_{\mathbb{R}^{2 D}} f(z) \PiBG(dz)\\
 &  + \int_0^t \left( \int_{\mathbb{R}^{2 D}}  \E_z L f(Z_s)  \PiBG(dz) \right)  ds
 \end{aligned}
\end{equation}
where we used Fubini's theorem to interchange time and space integrals.  This interchange (and subsequent ones) are valid since the function $L f$ is bounded by the Lyapunov function, which is an integrable function under the measure $\PiBG$.   We next argue that the second term on the right hand side of \eqref{eq:averaged_dynkin} vanishes due to Lemma~\ref{prop:infinitesimal_invariance} and some basic properties of Hamiltonian flows (volume and Hamiltonian preserving) and the momentum randomization map (Boltzmann-Gibbs preserving).


For this purpose, and as in Lemma~\ref{lemma:rhmc_drift_condition}, let $\{ t_i \}$ denote a realization of the sequence of random times at which the momentum randomizations occur.  For any $t>0$, let $\vartheta_t: \mathbb{R}^{2D} \to \mathbb{R}^{2D}$ be the Hamiltonian flow map associated to $H$.  We recall that the jump times and momentum randomizations are mutually independent, and that the number of jumps in any finite time interval is a.s.~finite.  By change of variables, and using the volume and Hamiltonian preserving properties of the Hamiltonian flow $\vartheta_s$, note that: \begin{equation} \label{eq:hamiltonian_part}
\int_{\mathbb{R}^{2D}} L f(\vartheta_s(z) ) \PiBG(dz) = \int_{\mathbb{R}^{2D}}  L f(z) \PiBG(dz) 
\end{equation} holds for any $s \in [0,t]$.  In addition, since $\PiBG$ is an invariant measure for the momentum randomization map and $Lf \in C^{\infty}_b(\mathbb{R}^{2D})$, we have the identity: \begin{equation} \label{eq:gamma_part}
\int_{\mathbb{R}^{2D}} \E L f(\Gamma(z) ) \PiBG(dz) = \int_{\mathbb{R}^{2D}}  L f(z) \PiBG(dz) \;.
\end{equation} These facts motivate us to decompose the process $Z_s$ into its Hamiltonian and momentum randomization pieces.  To do this, and with a slight abuse of notation, we regard the process $Z_s: \mathbb{R}^{2D} \to \mathbb{R}^{2D}$ as an evolution operator and decompose it via \[
Z_s(z) = \begin{cases}
\vartheta_{s-t_k}(Z_{t_k}) & \text{if $t_{k} \le s < t_{k+1}$} \\
\Gamma(\vartheta_{t_{k+1}-t_k}(Z_{t_k})) & \text{if $s = t_{k+1}$} 
\end{cases}
\] for all $s \ge 0 $.  To leverage this decomposition, we use $\{ t_i \}$ to split the time integral appearing in the second term of the right hand side of \eqref{eq:averaged_dynkin} into time intervals between consecutive momentum randomizations: \[
\E  \int_0^t  L f(Z_s(z))  ds =  \E \left\{ \sum_{k \ge 0} \int_{t_k \wedge t}^{t_{k+1} \wedge t}  L f(\theta_s \circ \Gamma \circ \theta_{\delta t_{k-1}} \circ \cdots \circ \theta_{\delta t_0}(z) ) ds \right\}
\] In this form, we can take advantage of the independence between momentum randomizations and jump times, in order to simplify this term.  In particular, we condition on the jump times and then average over individual momentum randomizations to obtain \begin{align*}
& \E  \int_0^t \int_{\mathbb{R}^{2D}}  L f(Z_s(z)) \PiBG(dz) ds  \\
&=  \E \left\{ \sum_{k \ge 0}   \int\limits_{t_k \wedge t}^{t_{k+1} \wedge t} \int\limits_{\mathbb{R}^{2D}}   L f(\theta_s \circ \Gamma \circ  \theta_{\delta t_{k-1}} \circ \cdots \circ \theta_{\delta t_0}(z) ) \PiBG(dz) ds  \right\} \\
&=  \E \left\{ \sum_{k \ge 0}  \int\limits_{t_k \wedge t}^{t_{k+1} \wedge t} \int\limits_{\mathbb{R}^{2D}}   L f(\theta_s \circ \Gamma \circ  \theta_{\delta t_{k-1}} \circ \cdots \circ 
\theta_{\delta t_1} \circ \Gamma(z) )  \PiBG(dz) ds \right\} \\
&=  \E \left\{ \sum_{k \ge 0}   \int\limits_{t_k \wedge t}^{t_{k+1} \wedge t} \int\limits_{\mathbb{R}^{2D}}   L f(\theta_s \circ \Gamma \circ  \theta_{\delta t_{k-1}} \circ \cdots \circ \theta_{\delta t_1}(z) ) \PiBG(dz) ds  \right\} \\
& = \cdots = \cdots  = \int_0^t \int_{\mathbb{R}^{2D}}  L f(z) \PiBG(dz) ds \;.
\end{align*}
where we sequentially used \eqref{eq:hamiltonian_part} and \eqref{eq:gamma_part} from initial time $0$ up to final time $s$.  Note that to use \eqref{eq:gamma_part} one has to average over the Gaussian random vector associated to the $i$th momentum randomization for $1 \le i \le k$ in the inner-most expectation.  Substituting this result back into \eqref{eq:averaged_dynkin} we obtain: \[
 \int_{\mathbb{R}^{2D}}   \E_z f(Z_t) \PiBG(dz) = \int_{\mathbb{R}^{2 D}} f(z) \PiBG(dz)  + t  \int_{\mathbb{R}^{2 D}}  L f(z)  \PiBG(dz) 
\]  To finish, we invoke Lemma~\ref{prop:infinitesimal_invariance}, which implies that $\PiBG$ is infinitesimally invariant for $L$, and hence, the second term on the right hand side of this equation vanishes, as required.   
\end{proof}

\subsection{Minorization Condition}

For $t \ge 0$, let $P_t$ denote the transition semigroup of the Markov process $Z_t$
 \[
P_t f(z) = \E f(Z_t(z)),\qquad Z_0(z) = z
\]
and let $\Pi_{t,z}$ denote the associated transition probability distribution \[
P_t f(z) = \int_{\mathbb{R}^{2D}}  f(w) \Pi_{t,z}(dw) \;.
\]
Recall that the process $Z_t$ only moves by either the Hamiltonian flow for a random duration or  momentum randomizations that are instantaneous.
Thus, we expect the semigroup to not have the strong Feller property \cite{DaZa1996}, since it lacks a sufficient regularizing effect.
Nevertheless, we can prove a minorization condition for this process by using the weaker regularizing effect of the momentum randomizations.

\begin{prop} \label{prop:minorization}
Suppose Hypothesis~\ref{hypothesis_on_phi} holds.
For every compact set $\Omega \subset \mathbb{R}^{2D}$,  there exist a probability measure $\eta$ over $\Omega$, $\epsilon>0$ and $t>0$ such that: \[
\Pi_{t,z}(\cdot) \ge \epsilon \eta(\cdot)
\] holds for all $z \in \Omega$.
\end{prop}

To prove this proposition, it is convenient to introduce the following operator: \begin{equation} \label{eq:Aoperator}
\mathcal{A} f(z) =  \E f(\Gamma(z)) =  \int_{\mathbb{R}^{D}} f(q, \eta) g_{\phi}(p, \eta) d \eta
\end{equation}
where $g_{\phi}: \mathbb{R}^{2D} \to \mathbb{R}^+$ is defined as: \begin{equation} \label{eq:Atransitiondensity}
g_{\phi}(p, \eta) = (2 \pi \sin^2\phi )^{-D/2} \exp\left( - \frac{ | \eta - p \cos \phi |^2 }{2 \sin^2 \phi} \right) \;.
\end{equation}
Note that the operator $\mathcal{A}$ appears in the infinitesimal generator in \eqref{eq:generator}.  For any $t\ge0$, let $\theta_t: \mathbb{R}^{2D} \to \mathbb{R}^{2D}$ denote the Hamiltonian flow associated with  \eqref{eq:hamiltonian} and $\mathcal{C}_{\theta_t}$ denote the composition operator for $\theta_t$ defined as: \[
\mathcal{C}_{\theta_t} f(z) = f( \theta_t(z) ) \;.
\]
This Hamiltonian flow can be characterized by a generating function $S_t(q_0, q_1)$ \cite{SaCa1994, MaRa1999, MaWe2001}. Specifically,
if $|q_1 - q_0|$ and $t>0$ are sufficiently small, then $ (q_1, p_1) = \theta_t(q_0, p_0)$ satisfy the following system of equations \begin{equation} \label{eq:discrete_hamiltons_equations}
\begin{aligned}
p_0 &= -D_1 S_t(q_0, q_1) \\
p_1 &= D_2 S_t(q_0, q_1)
\end{aligned}
\end{equation}
Here $D_i$ denotes the derivative with respect to the $i$th component of $S_t$.
Moreover, the generating function can be written as \begin{equation}  \label{eq:generating_function}
S_t(q_0, q_1) = \int_0^t \left( \frac{1}{2} | \dot{q}(s) |^2 - \Phi(q(s)) \right) ds
\end{equation} where $q: [0, t] \to \mathbb{R}^D$ solves the Euler-Lagrangian equations $\ddot q = - \nabla \Phi(q)$ with boundary conditions $q(0) = q_0$ and $q(t) = q_1$.
In discrete mechanics, the equations \eqref{eq:discrete_hamiltons_equations} and the generating function $S_t$ are known as discrete Hamilton's equations and the (exact) discrete Lagrangian, respectively \cite{MaWe2001}.

\begin{proof}
 We adapt to our setting some ideas from the proof of Theorem 2.3 in \cite{ELi2008}.   To establish the desired result, we use the (weak) regularizing effect of the operator $\mathcal{A}$ on a function $f$ in the momentum degrees of freedom.  Since the Hamiltonian flow is regular \cite{MaRa1999}, this regularizing effect can be transferred to the position degrees of freedom of $f$.  Similar results appear in Lemma 2.2 of \cite{ELi2008} and the proof of Theorem~2.2 of \cite{BoOw2010}.


Specifically, a change of variables shows that: \begin{equation} \label{eq:ACAtransitiondensity}
(\mathcal{A} C_{\theta_t} \mathcal{A}  f)(q,p) = \int_{\mathbb{R}^D} \int_{\mathbb{R}^D} f(q_1, p_1) g_{t,\phi}((q,p), (q_1, p_1)) dq_1 dp_1
\end{equation}
where we have introduced the transition density of the operator $\mathcal{A} C_{\theta_t} \mathcal{A}$: \[
g_{t,\phi}((q,p), (q_1, p_1))  =  |\det D_{12} S_t(q,q_1)|  g_{\phi}(p,  -D_1 S_t(q, q_1)) g_{\phi}( D_2 S_t(q, q_1), p_1)
\]
in terms of $g_{\phi}$ in \eqref{eq:Atransitiondensity} and $S_t$ in \eqref{eq:generating_function}.



To take advantage of \eqref{eq:ACAtransitiondensity}, we consider the following Duhamel formula: \begin{align*}
\left. e^{\lambda^{-1} (s-t)} \mathcal{C}_{\theta_{t-s}} P_s f( z ) \right|_{s=0}^{s=t}
&= P_t f(z) - e^{- \lambda^{-1} t} \mathcal{C}_{\theta_t} f( z)  \\
&= \lambda^{-1} \int_0^t e^{\lambda^{-1} (s-t)}  \mathcal{C}_{\theta_{t-s}} \mathcal{A} P_s f ( z )   ds
\end{align*}
A second application of this Duhamel formula implies that \begin{align*}
& P_t f(z) =
e^{- \lambda^{-1} t} \mathcal{C}_{\theta_{t}}  f(z)
+ \lambda^{-1} \int_0^t e^{- \lambda^{-1} t} \mathcal{C}_{\theta_{t-t_1}}  \mathcal{A} \mathcal{C}_{\theta_{t_1}}  f ( z)  dt_1 \\
&\qquad +\lambda^{-2} \int_0^t \int_0^{t_1} e^{\lambda^{-1} (t_1-t)}   e^{\lambda^{-1} (t_2-t_1)}  C_{\theta_{t-t_1}} \mathcal{A} C_{\theta_{t_1-t_2}}   \mathcal{A} P_{t_2} f ( z )   dt_2 dt_1
\end{align*}
A third application of this formula yields \begin{align}
P_t f(z) &\ge  \lambda^{-2} \int_0^t \int_0^{t_1} e^{\lambda^{-1} (t_2-t)}   C_{\theta_{t-t_1}}  \mathcal{A} C_{\theta_{t_1-t_2}}  \mathcal{A}  \mathcal{C}_{\theta_{t_2}}  f ( z )   dt_2 dt_1 \nonumber \\
&\ge \lambda^{-2} \int_{\epsilon t}^t \int_0^{t_1-\epsilon t_1} e^{\lambda^{-1} (t_2-t)}   C_{\theta_{t-t_1}}  \underbrace{\mathcal{A} C_{\theta_{t_1-t_2}}  \mathcal{A}}  \mathcal{C}_{\theta_{t_2}}  f ( z )   dt_2 dt_1 \label{eq:lower_bound_on_Pt}
\end{align}
for $\epsilon$ sufficiently small.   Since $\epsilon t \le t_1 \le t$ and $0 \le t_2 \le t_1 - \epsilon t_1$, we have that $ \epsilon^2 t \le t_1 - t_2 \le t$.
Combining this result with \eqref{eq:ACAtransitiondensity} applied to the bracketed term in \eqref{eq:lower_bound_on_Pt}, yields the desired result.

\end{proof}

%

Next we show that Hypothesis~\ref{minorization} of Harris Theorem holds for the transition probabilities of the RHMC method.
To state this Proposition, we will use the total variation (TV) distance between  measures. Recall that
\begin{equation}
\|\mu - \nu\|_\TV = 2\sup_A |\mu(A) - \nu(A)|\;,
\end{equation}
where the supremum runs over all measurable sets. In particular, the total variation distance
between two probability measures is two if and only if they are mutually singular.

%
%

\begin{lemma} \label{lemma:minorization}
Suppose Hypothesis~\ref{hypothesis_on_phi} holds and let $V(z)$ be the Lyapunov function from Lemma~\ref{lemma:rhmc_drift_condition}.
For every $E>0$, there exist $t>0$ and $\epsilon>0$ such that: \[
\| \Pi_t(z_1, \cdot) - \Pi_t(z_2, \cdot) \|_{\TV} \le 2 (1- \epsilon)
\] holds for all $z_1, z_2 \in \mathbb{R}^{2D}$ satisfying $V(z_1) \vee V(z_2) <E$.
\end{lemma}

\begin{proof}
Proposition~\ref{prop:minorization} implies the
following transition probability $\tilde\Pi_{t,z}$ is well-defined: \[
\tilde\Pi_{t,z}(\cdot)
= \frac{1}{1-\epsilon} \Pi_{t,z}(\cdot) - \frac{\epsilon}{1-\epsilon} \eta(\cdot)
\]
for any $z$ satisfying $V(z)<E$.   Therefore, \[
\| \Pi_{t,z_1}(\cdot)- \Pi_{t,z_2}(\cdot) \|_{\TV} =
(1-\epsilon) \|  \tilde\Pi_{t,z_1}(\cdot) - \tilde\Pi_{t,z_2}(\cdot) \|_{\TV}
\]
for all $z_1,z_2$ satisfying
$V(z_1) \vee V(z_2) < E$.
Since the TV norm is bounded by $2$, one obtains the desired result.
\end{proof}

\subsection{Main Result: Geometric Ergodicity}

With a Lyapunov function and minorization condition in hand, we are now in position to state a main result of this paper.

\begin{theorem}
Suppose Hypothesis~\ref{hypothesis_on_phi} holds and let $V(z)$ be the Lyapunov function from Lemma~\ref{lemma:rhmc_drift_condition}.
Then the Markov process induced by $L$ has a unique invariant probability measure given by $\PiBG$.  Furthermore, there exist positive constants $r$ and $C$ such that \begin{equation} \label{eq:ell1bound}
 \| \Pi_{t,z} - \PiBG \|_{\TV} \le C \; V(z) \; e^{- r t}
\end{equation}
for  all $t \ge 0$ and all $z \in \mathbb{R}^{2D}$.
\end{theorem}

\begin{proof}
 According to Lemma~\ref{lemma:rhmc_drift_condition}, the generator $L$ satisfies a Foster-Lyapunov Drift Condition.  Moreover, its associated transition probabilities satisfy Lemma~\ref{lemma:minorization}.  Hence, Theorem~\ref{uncountable_harris} implies that (i) the process possesses a unique invariant distribution, and (ii) the transition probability of the process converges to this invariant distribution geometrically fast in the TV metric.  Since, by Prop.~\ref{prop:infinitesimal_invariance}, $\PiBG$ is an infinitesimally invariant measure for the process, and that compactly supported functions are in the domain of $L$, it follows that $\PiBG$ is the unique invariant measure for the process.
\end{proof}

%
%

\section{Model Problem} \label{sec:model_problem}

For simplicity, we assume in this section and the next that the Horowitz angle is chosen as $\phi=\pi/2$, i.e., the momentum randomizations are complete rather than partial.

We quantify the sampling capabilities of the RHMC method given in Algorithm~\ref{algo:rhmc} on a model  problem in which the target is a multivariate Gaussian distribution with uncorrelated components, some of them with small variances.  We also compare against the standard HMC method given in Algorithm~\ref{algo:hmc}, where the duration is a fixed parameter.  This model problem is discussed in \cite{Ne2011} and analyzed further in \cite{BlCaSa2014}.  This distribution can be interpreted as a truncation of an infinite-dimensional normal distribution on a Hilbert space \cite{BePiSaSt2011}. The potential energy function is given by: \begin{equation} \label{eq:potential_energy_gaussian}
\Phi(q_1, \cdots, q_D) = \sum_{i=1}^D \frac{1}{2 \sigma_i^2} q_i^2 \;,
\end{equation} where $\sigma_i^2$ is the variance in the $i$th component.

\subsection{The process}



For \eqref{eq:potential_energy_gaussian}, a sample trajectory of the RHMC method satisfies
 \begin{equation} \label{eq:ith_sde}
q_i(t_{n+1}) = \cos\left( \frac{\delta t_n}{\sigma_i} \right) q_i(t_n) + \sigma_i \sin\left( \frac{\delta t_n}{\sigma_i} \right) \xi_{i, n}  
\end{equation} with $q_i(0)$ given and where $\{ t_n \}$ are random jump times related by \[
t_n = t_{n-1} + \delta t_{n-1}   \;,
\] with $t_0=0$. Here we have introduced the following sequences of random variables \[
 \{ \xi_{i,n} \} \overset{\text{iid}}{\sim} \mathcal{N}(0,1)  \;, \quad \{ \delta t_n \} \overset{\text{iid}}{\sim} \Exp\left(\frac{1}{\lambda} \right)
\] where $i$ (resp.~$n$) runs over the components of the target distribution (resp.~jump times).  


The  solution of the stochastic difference equation \eqref{eq:ith_sde} is given by: \begin{equation} \label{eq:ith_solution}
q_i(t_n) = \prod_{j=0}^{n-1} \cos\left( \frac{\delta t_j}{\sigma_i} \right) q_i(0) + \sum_{j=0}^{n-1} \sigma_i \sin\left( \frac{\delta t_j}{\sigma_i} \right) \prod_{k=j+1}^{n-1} \cos\left(\frac{\delta t_k}{\sigma_i} \right) \xi_{i, j} 
\end{equation}
(We adhere to the standard convention that  $\prod_{k=m}^n \cdot $ takes the value 1 if $m>n$.)  
At steady-state the $i$th component of this process is normally distributed with mean zero and variance $\sigma_i^2$, i.e., \[
q_i(t) \overset{\text{d}}{\to} \mathcal{N}(0,\sigma_i^2) \quad \text{as $t \to \infty$} \;.
\]
The corresponding solution for HMC is given by formula \eqref{eq:ith_solution} with constant, as opposed to random, durations $\delta t_n$.  Note that,  for constant duration, formula (\ref{eq:ith_sde}) makes apparent that, at stationarity, the correlation between $q_i(t_{n+1})$ and $q_i(t_n)$ is $\cos(\delta t_n/\sigma_i)$. When $\lambda$ is an even integer multiple of $\pi\sigma_i$, $q_i(t_{n+1})=q_i(t_n)$; for odd multiples,
$q_i(t_{n+1})=-q_i(t_n)$. For those value of $\lambda$ the chain is not ergodic; for values of $\lambda$ close to an integer multiple of $\pi\sigma_i$, the performance of the chain may be expected to be poor.

\subsection{Integrated Autocorrelation Time (IAC)}

The first measure of the quality of the samples provided by the algorithm, we consider is the integrated autocorrelation time (IAC) associated with estimating the mean of the target distribution.
The natural estimator for this purpose is: \[
\hat{f}_{i,N} =  \frac{1}{N+1} \sum_{j=0}^{N} q_i(t_j) \;.
\]
As shown in, e.g., Chapter IV of \cite{AsGl2007}, asymptotically as $n\rightarrow \infty$, the variance of the estimator behaves as
 $\hat{\sigma}_i^2 /n$, where the constant $\hat{\sigma}_i^2$  is called the asymptotic variance.
This may be computed by means of the formula
\begin{equation}\label{eq:series}
\hat{\sigma}_i^2 =  \Var(q_i(0)) + 2 \sum_{j=1}^{\infty} \cov(q_i(0), q_i(t_j) ) \;.
\end{equation}
This  holds for any random variables $q_i(t_j)$  for which both
 $\E(q_i(t_j))$ and $\E(q_i(t_j)q_i(t_{j+k}))$ are independent of $j$  (it is not necessary, in particular, that the $q_i(t_j)$ originate from a Markov chain),  provided that the series converges absolutely.
 If successive samples are independent, then the series  vanishes and the asymptotic variance equals the variance $\sigma_i^2$ of $q_i(0)$. The integrated autocorrelation time (IAC) of the estimator $\hat{f}_{i,N}$ is defined as the ratio $\hat{\sigma}_i^2 / \sigma_i^2$.   It follows from \eqref{eq:ith_solution} that,
 at stationarity of the chain,
 \begin{align*}
 \cov(q_i(0), q_i(t_j) )  = \E (q_i(0)  q_i(t_j))
 &= \sigma_i^2  \E \prod_{k=0}^{j-1}  \cos\left( \frac{\delta t_k}{\sigma_i} \right) \\
 &=  \sigma_i^2  \prod_{k=0}^{j-1} \E \cos\left( \frac{\delta t_k}{\sigma_i} \right) = \sigma_i^2 \left( \frac{\sigma_i^2}{\sigma_i^2+\lambda^2} \right)^j
\end{align*}
and, hence,
\begin{equation} \label{eq:rhmc_iac}
\IAC_i^{\RHMC}  = 1 + 2 \frac{\sigma_i^2}{\lambda^2} \;.
\end{equation}
Thus, the IAC in each component monotonically decreases with increasing $\lambda$ and then plateaus at unity.

A very similar calculation shows that the corresponding formula for the HMC method is given by
\begin{equation} \label{eq:hmc_iac}
\IAC_i^{\HMC} =  \frac{1+\cos\left( \dfrac{\lambda}{\sigma_i} \right)}{1-\cos\left( \dfrac{\lambda}{\sigma_i} \right)} \;.
\end{equation}  Unlike \eqref{eq:rhmc_iac}, the IAC for the HMC method is an oscillatory function of $\lambda$: as discussed in the introduction it is possible that
increasing $\lambda$ (and therefore increasing computational costs) results in higher correlation.

It is useful to discuss the following choices of $\lambda$:
\begin{itemize}
\item  $\lambda$ is an even multiple of $\pi\sigma_i$ leading to $q_i(t_{n+1})=q_i(t_n)$.  The chain is not ergodic and formula \eqref{eq:series} is not valid, due to the divergence of the series. Note however that the estimator yields the value $q_i(0)$;
     when $q(0)$ is taken from the stationary distribution,  the variance of the estimator is therefore $\sigma_i^2$, regardless of the number of samples. Thus the asymptotic variance is not finite, which agrees with \eqref{eq:hmc_iac}, even though this formula was derived from \eqref{eq:series}, invalid in this case.
\item  $\lambda$ is an even multiple of $\pi\sigma_i$, with $q_i(t_{n+1})=-q_i(t_n)$. The chain is not ergodic
    and formula \eqref{eq:series} is not valid, due to the divergence of the series. However, for  the particular statistics being considered, the estimation is exceptionally good.
    If the number of samples is even, the estimator gives the exact value 0, and, for an odd number, the estimator yields the very accurate  value $q_i(0)/(N+1)$ (with an $\mathcal{O}(N^{-1})$, rather than
    $\mathcal{O}(N^{-1/2})$ error). Therefore the asymptotic variance vanishes, in agreement with
    \eqref{eq:hmc_iac}.
\end{itemize}

\subsection{Mean Squared Displacement (MSD)}

As another metric for the quality of sampling, we consider  the single-step, equilibrium mean-squared displacement of the RHMC and HMC methods.  This statistic quantifies how far apart successive samples are.  A direct calculation using \eqref{eq:ith_sde} shows that \begin{align*}
\MSD^{\RHMC} &= \sum_{i=1}^D \E |q_i(t_1) - q_i(0)|^2 \\
&= \sum_{i=1}^D \E \left\{  \left| (\cos\left( \frac{t_1}{\sigma_i} \right) - 1) q_i(0) + \sigma_i \sin\left( \frac{t_1}{\sigma_i} \right) \right|^2 \right\}  \\
&= \sum_{i=1}^D \E \left\{ ( \cos\left( \frac{t_1}{\sigma_i} \right) - 1)^2  + \sin^2 \left( \frac{t_1}{\sigma_i} \right) \right\}  \sigma_i^2   \\
&=  \sum_{i=1}^D \E \left\{ 2 - 2 \cos\left( \frac{t_1}{\sigma_i} \right) \right\} \sigma_i^2,
\end{align*}
which implies that: \begin{equation} \label{eq:rhmc_msd}
  \MSD^{\RHMC} = \sum_{i=1}^D \frac{2 \lambda^2 \sigma_i^2}{\sigma_i^2+\lambda^2}   \end{equation}
Note that $\MSD^{\RHMC}$ monotonically increases with increasing $\lambda$ and then plateaus at $\sum_{i=1}^D 2 \sigma_i^2$.  Since in the present scenario of small time steps, the computational cost of the algorithm is proportional to $\lambda$, the function $\MSD^{\RHMC}/\lambda$ may be taken as a measure of the efficiency of the algorithm. For $\lambda$ small this function increases with $\lambda$; it reaches a maximum at
\begin{equation}\label{eq:efficiency}
\lambda_{\max}= \left( \frac{\sum_{i=1}^D\sigma_i^4} {\sum_{i=1}^D\sigma_i^2} \right)^{1/2}
\end{equation}
and then decreases. The quantity \eqref{eq:efficiency} is approximately $\max_{1\leq i\leq D} \sigma_i$ and the conclusion is that, from this point of view, the best value of $\lambda$ coincides with the standard deviation of the least constrained variate. Taking $\lambda$ above the optimal value will not decrease the mean square displacement (as distinct from the situation for HMC we discuss next), but will waste computational resources.

A similar calculation shows that the single-step, equilibrium mean-squared displacement  of the  HMC is given by \begin{equation} \label{eq:hmc_msd}
\MSD^{\HMC} = \sum_{i=1}^D 2 \left( 1- \cos\left( \dfrac{\lambda}{\sigma_i} \right) \right) \sigma_i^2
\end{equation}
which is an oscillatory function of $\lambda$.

\section{Numerical Testing}
\label{sec:numerics}

\subsection{Standard Normal Distribution}

We first numerically illustrate the analysis in Section~\ref{sec:model_problem} assuming $D=1$ and unit variance.

Figure~\ref{fig:iac_gaussian_1d} corresponds to $\IAC$. In the left panel (HMC),
for $\lambda$ close to $\pi$ or $3\pi$ there is a clear discrepancy between  the function in \eqref{eq:hmc_iac} and the empirical estimate based on a single trajectory: due to the divergence of the series \eqref{eq:series} the software package used to measure  $\IAC$ does not work satisfactorily.
The right panel corresponds to RHMC.  Note the horizontal asymptote at unity, which is consistent with \eqref{eq:rhmc_iac}.

Figure~\ref{fig:msd_gaussian_1d} shows results for  $\MSD$.   On the left (HMC), the lack of ergodicity at the values  $\pi$, $3\pi$ entails, in spite of the large number ($10^6$) of samples,  discrepancies between the value at stationarity in (\ref{eq:hmc_msd}) and the empirical value along the trajectory considered. We observe the monotonic behavior on the right (RHMC) with a horizontal asymptote at 2.


\begin{figure}
\begin{center}
\includegraphics[width=0.45\textwidth]{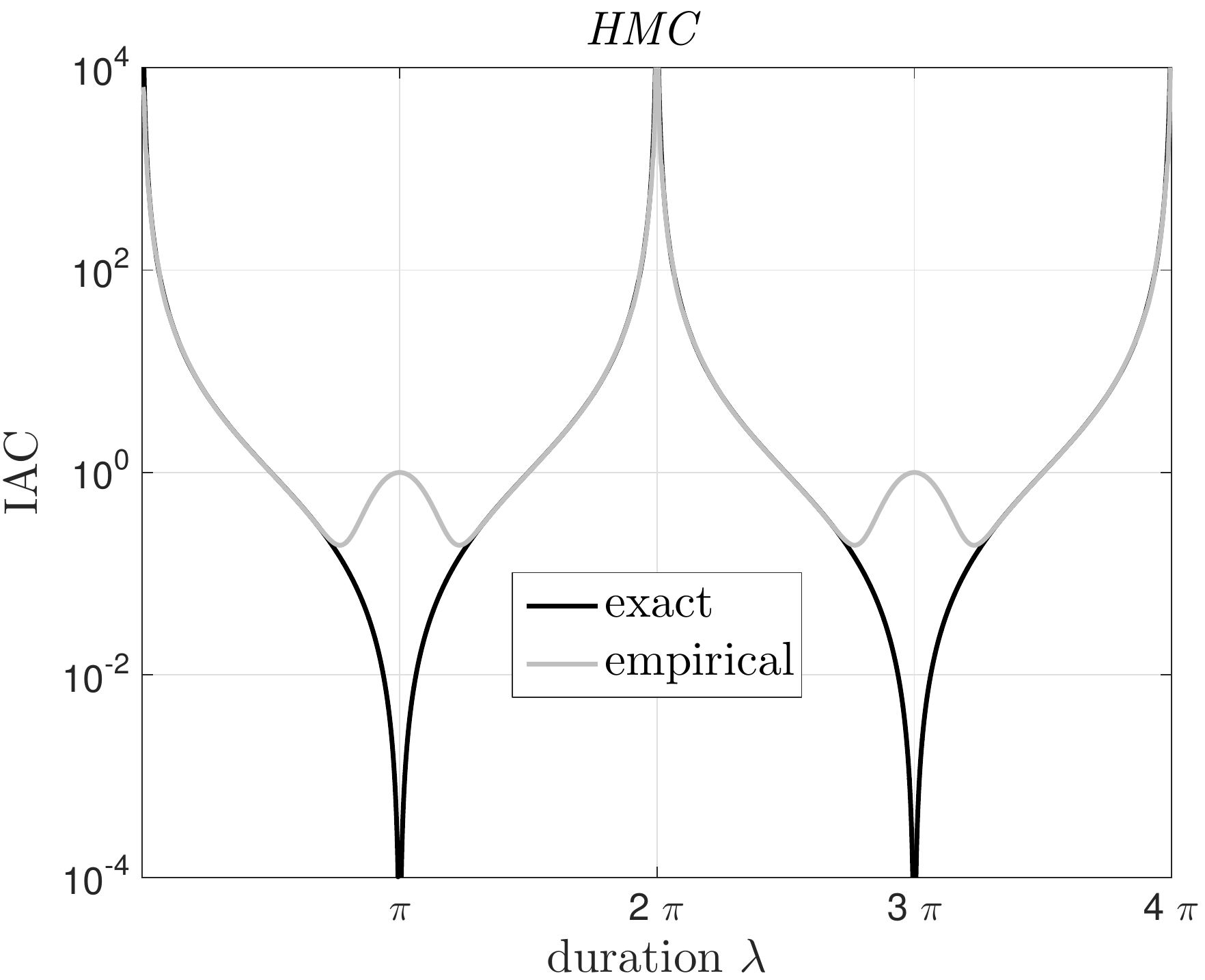}
\includegraphics[width=0.45\textwidth]{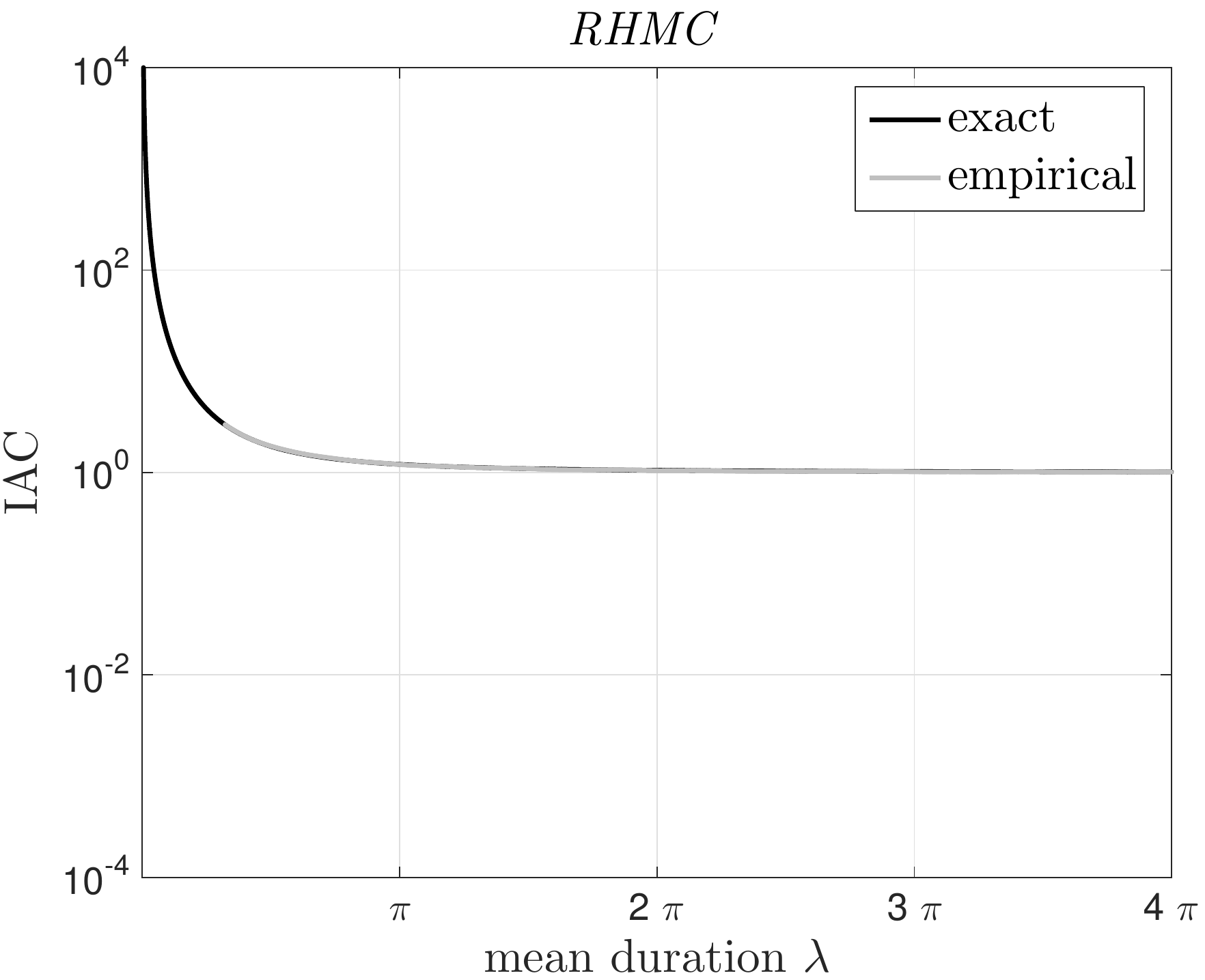}
\end{center}
\caption{ \small {\bf IAC for Normal Distribution.}
The left (resp.~right) panel of this figure plots $\IAC^{\HMC}$ (resp.~$\IAC^{\RHMC}$) vs.~duration (resp.~mean duration) $\lambda$ for the HMC (resp.~RHMC) method applied to a standard normal distribution.    The black lines in the left (resp.~right) panel plot $\IAC^{\HMC}$ (resp.~$\IAC^{\RHMC}$) as given in \eqref{eq:hmc_iac} (resp.~\eqref{eq:rhmc_iac}).  The grey lines show an empirical estimate obtained by using an output trajectory with $10^6$ samples and the ACOR software package \cite{Go2009, So1997}.  
 }
 \label{fig:iac_gaussian_1d}
\end{figure}

\begin{figure}
\begin{center}
\includegraphics[width=0.45\textwidth]{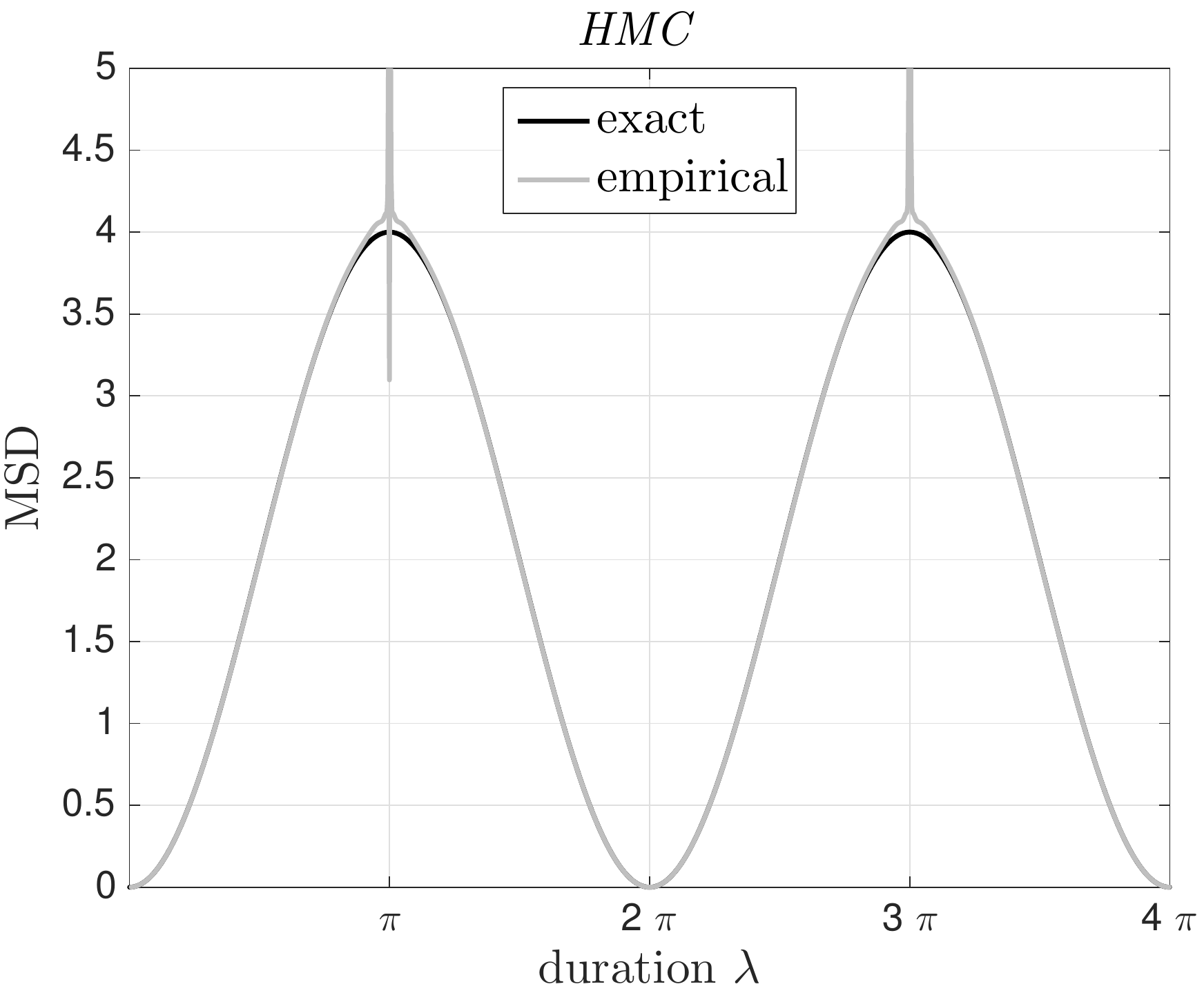}
\includegraphics[width=0.45\textwidth]{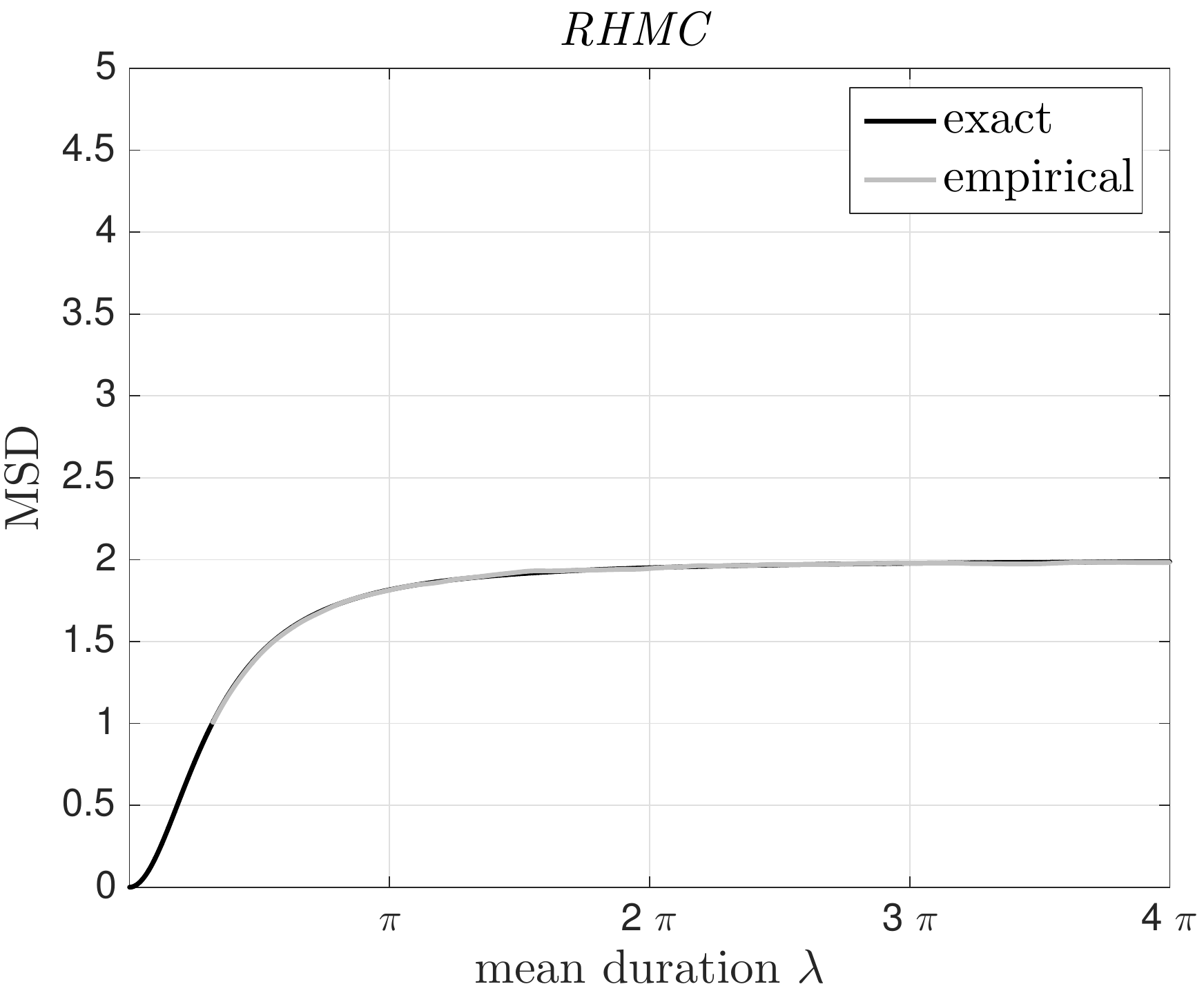}
\end{center}
\caption{ \small {\bf MSD for Normal Distribution.}
The left (resp.~right) panel of this figure plots $\MSD^{\HMC}$ (resp.~$\MSD^{\RHMC}$) vs.~duration (resp.~mean duration)  $\lambda$ for the HMC (resp.~RHMC) method applied to a standard normal distribution.   The black lines in the left (resp.~right) panel plot $\MSD^{\HMC}$ (resp.~$\MSD^{\RHMC}$) as given in \eqref{eq:hmc_msd} (resp.~\eqref{eq:rhmc_msd}). The grey lines show an empirical estimate obtained by using an output trajectory with $10^6$ samples.
}
 \label{fig:msd_gaussian_1d}
\end{figure}


\subsection{Ten-Dimensional Normal Distribution}

 We next consider the case $D=10$ and 
\[
\sigma_i = \frac{i}{D}, \qquad 1 \le i \le D \;.
\]

\noindent
Figure~\ref{fig:iac_gaussian_nd} compares the IAC of HMC and RHMC.
Figure~\ref{fig:msd_gaussian_nd} refers to MSD. The pathologies of HMC in the one-dimensional case discussed above are also manifest here, but, with 10 different frequencies in the Hamiltonin system, there are more resonant choices of $\lambda$ leading to lack of ergodicity.
In the right panel of Figure~\ref{fig:msd_gaussian_nd}, we see that $\MSD^{\RHMC}$ has a horizontal asymptote at \[
\sum_{i=1}^{10} 2 \sigma_i^2 \approx 7.7 \;.
\]  To summarize, the behavior of MSD of the HMC method as a function of the duration is complex, whereas for the RHMC method the dependence is monotonic and plateaus.


\begin{figure}
\begin{center}
\includegraphics[width=0.45\textwidth]{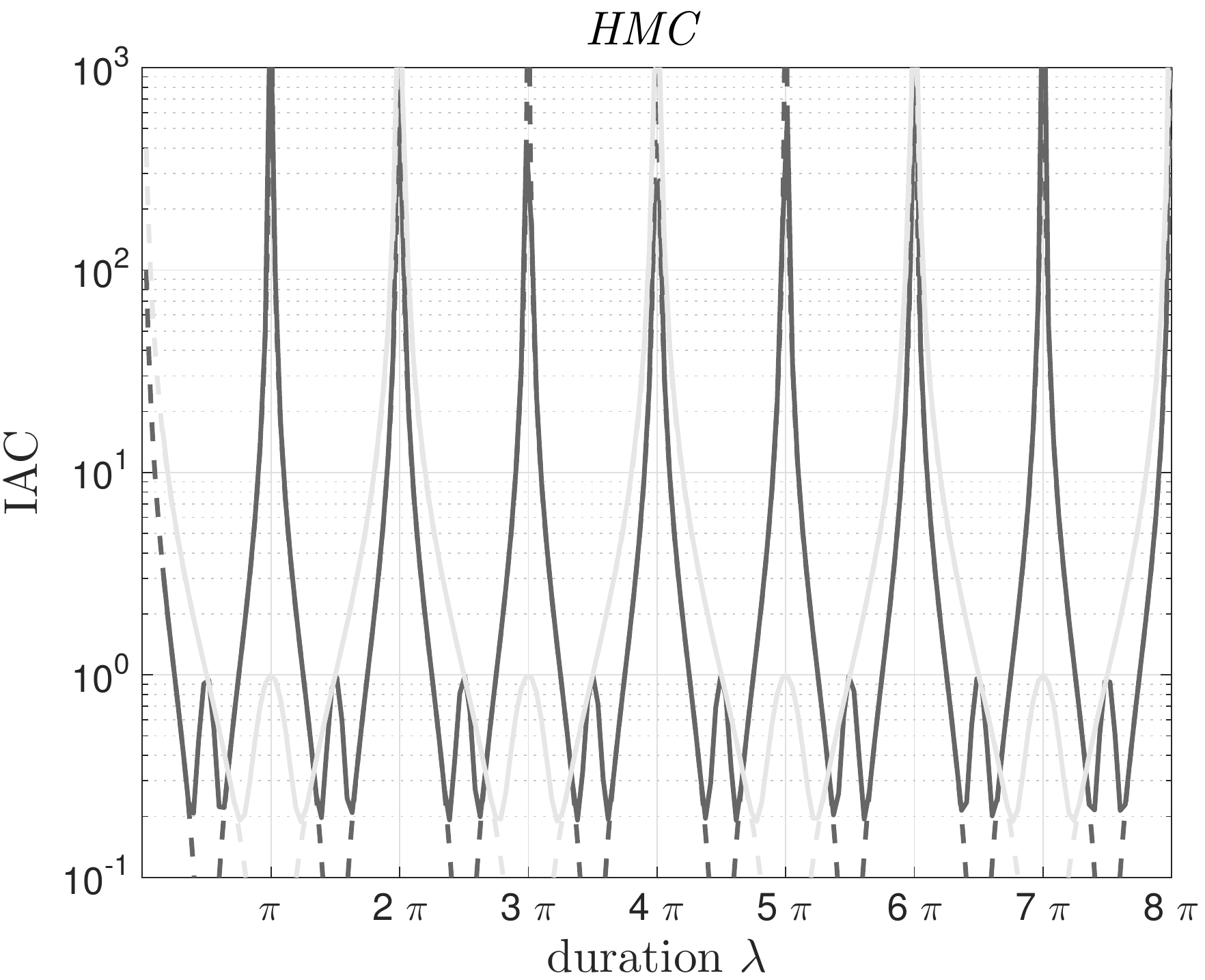}
\includegraphics[width=0.45\textwidth]{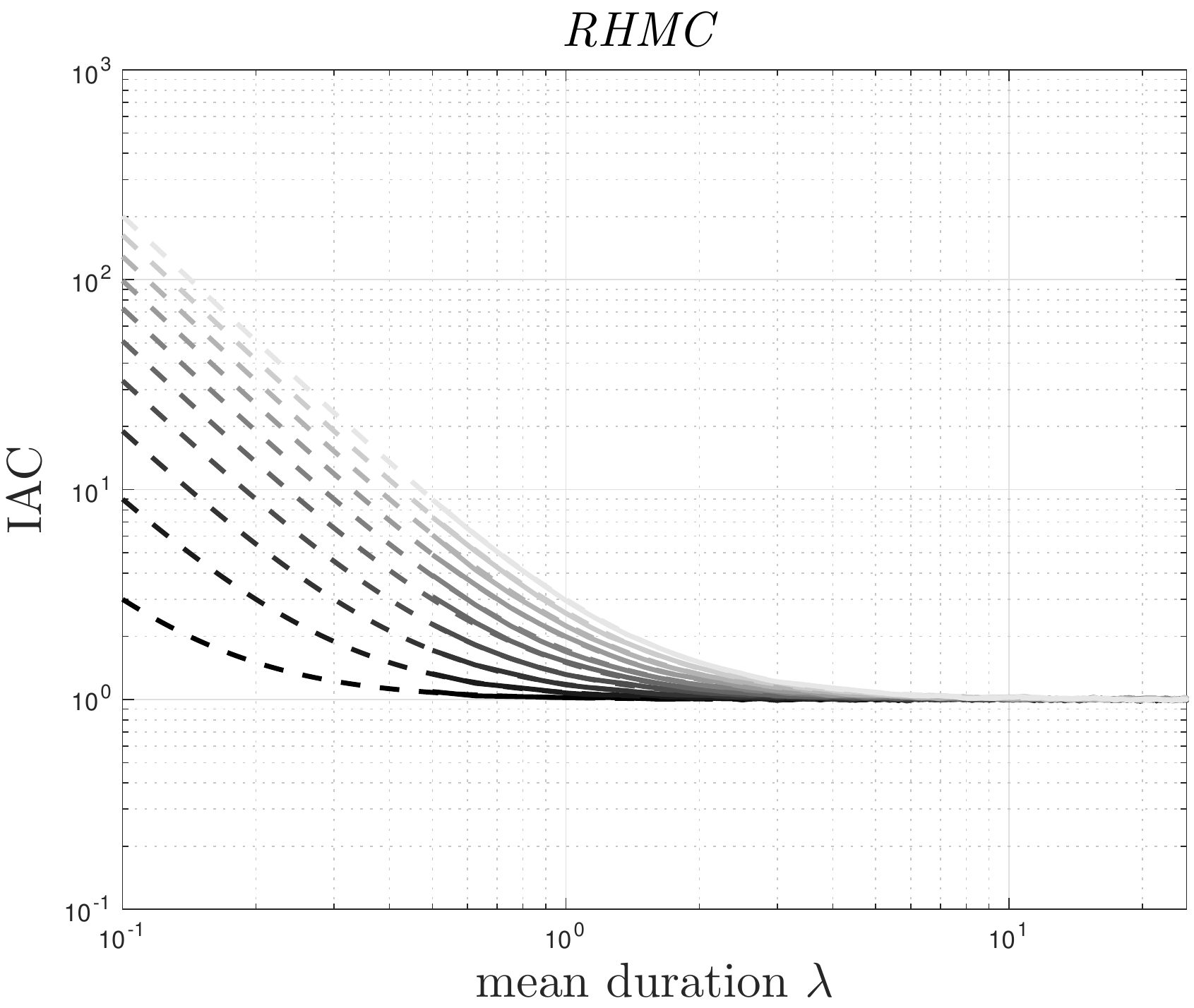}
\end{center}
\caption{ \small  {\bf IAC for Ten-Dimensional Normal Distribution.}
The left panel plots $\IAC^{\HMC}_i$ given in \eqref{eq:hmc_iac} as a function of the duration $\lambda$ with $i=10$ (dashed light grey) and $i=5$ (dashed dark grey); and an approximation of $\IAC^{\HMC}_i$ using a time series: $\{ q_i(t_k)   \}_{1 \le k \le 10^6}$ and the ACOR software package with $i=10$ (solid light grey) and $i=5$ (solid dark grey) \cite{Go2009, So1997}.
The dashed lines in the right panel plot $\IAC^{\RHMC}_i$ given in \eqref{eq:rhmc_iac} as a function of the mean duration $\lambda$ for $1 \le i \le 10$.
The solid lines plot an approximation of $\IAC^{\RHMC}_i$ using a time series: $\{ q_i(t_k)   \}_{1 \le k \le 10^6}$ and the ACOR software package \cite{Go2009, So1997}.
Different shades of grey in the right panel indicate different components of the target distribution, with darker shades corresponding to lower variance.   }
 \label{fig:iac_gaussian_nd}
\end{figure}

\begin{figure}
\begin{center}
\includegraphics[width=0.45\textwidth]{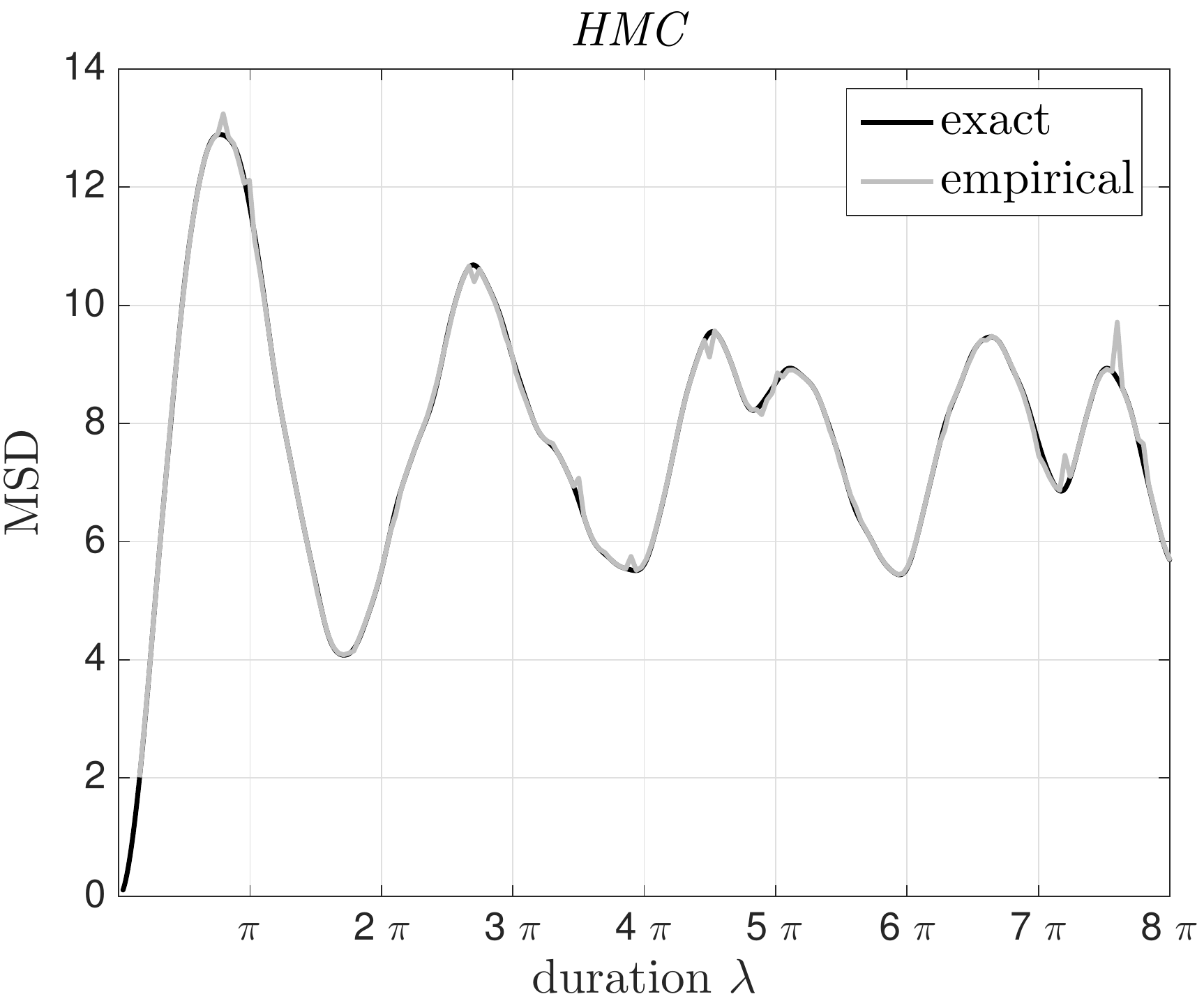}
\includegraphics[width=0.45\textwidth]{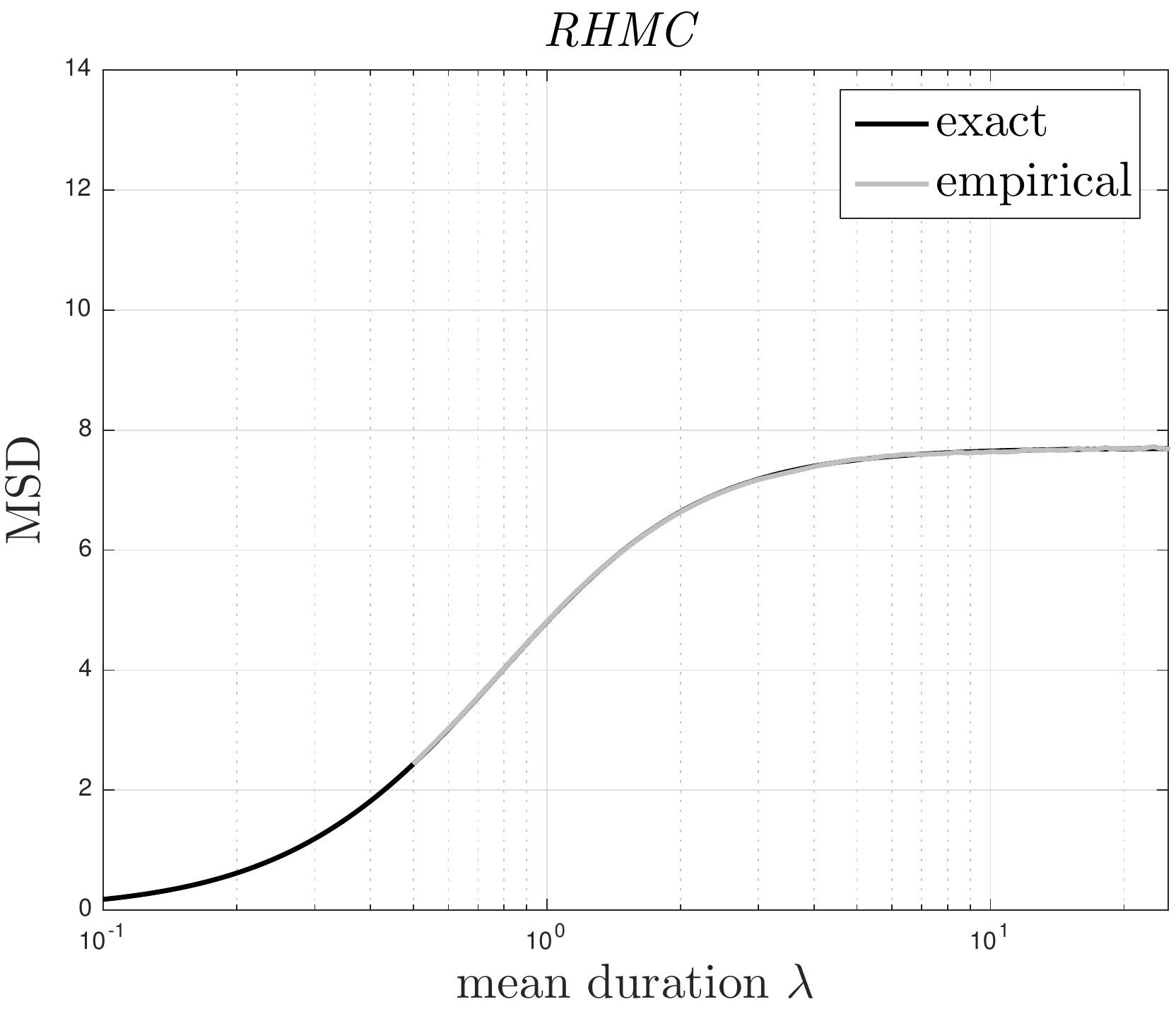}
\end{center}
\caption{ \small {\bf MSD for Ten-Dimensional Normal Distribution.}
The solid black lines in the left (resp.~right) panel plot the formulas given in \eqref{eq:hmc_msd} (resp.~\eqref{eq:rhmc_msd}) as a function of the duration (resp.~mean duration) $\lambda$.  The solid grey lines in the left (resp.~right) panel plot an approximation of  $\MSD^{\HMC}$ (resp.~$\MSD^{\RHMC}$) associated to the time series: $\{ q_i(t_k)   \}_{1 \le k \le 10^6}$ vs.~$\lambda$ for $1 \le i \le D$.   }
 \label{fig:msd_gaussian_nd}
\end{figure}


\subsection{Two-dimensional Double Well}

Consider a Brownian particle with the following two-dimensional double-well potential energy function \begin{equation} \label{eq:dw_2d}
\Phi(x_1, x_2)=5 (x_2^2-1)^2 + 1.25 \left( x_2 - \frac{x_1}{2} \right)^2   \;.
\end{equation} whose contours  are displayed in Figure~\ref{fig:dw_2d}.  The left (resp.~right) panel of Figure~\ref{fig:iac_dw_2d} plots the IAC $\tau$ vs duration $\lambda$ for estimating the mean of the time series $\{ f(Q_1(t_{i-1}),Q_2(t_{i-1})) \}_{1 \le i \le 10^6}$ where $f(x,y) = 2 x + y$ produced by a HMC (resp.~RHMC) scheme.  This test function is the dot product between the configuration vector $( Q_1(t_{i-1}),Q_2(t_{i-1}) )$ and the line connecting the two wells or $\operatorname{span}((2,1))$.  For the same output trajectory, the left (resp.~right) panel of Figure~\ref{fig:msd_dw_2d} plots an estimate of the equilibrium mean-squared displacement produced by a HMC (resp.~RHMC) scheme.


\begin{figure}
\begin{center}
\includegraphics[width=0.5\textwidth]{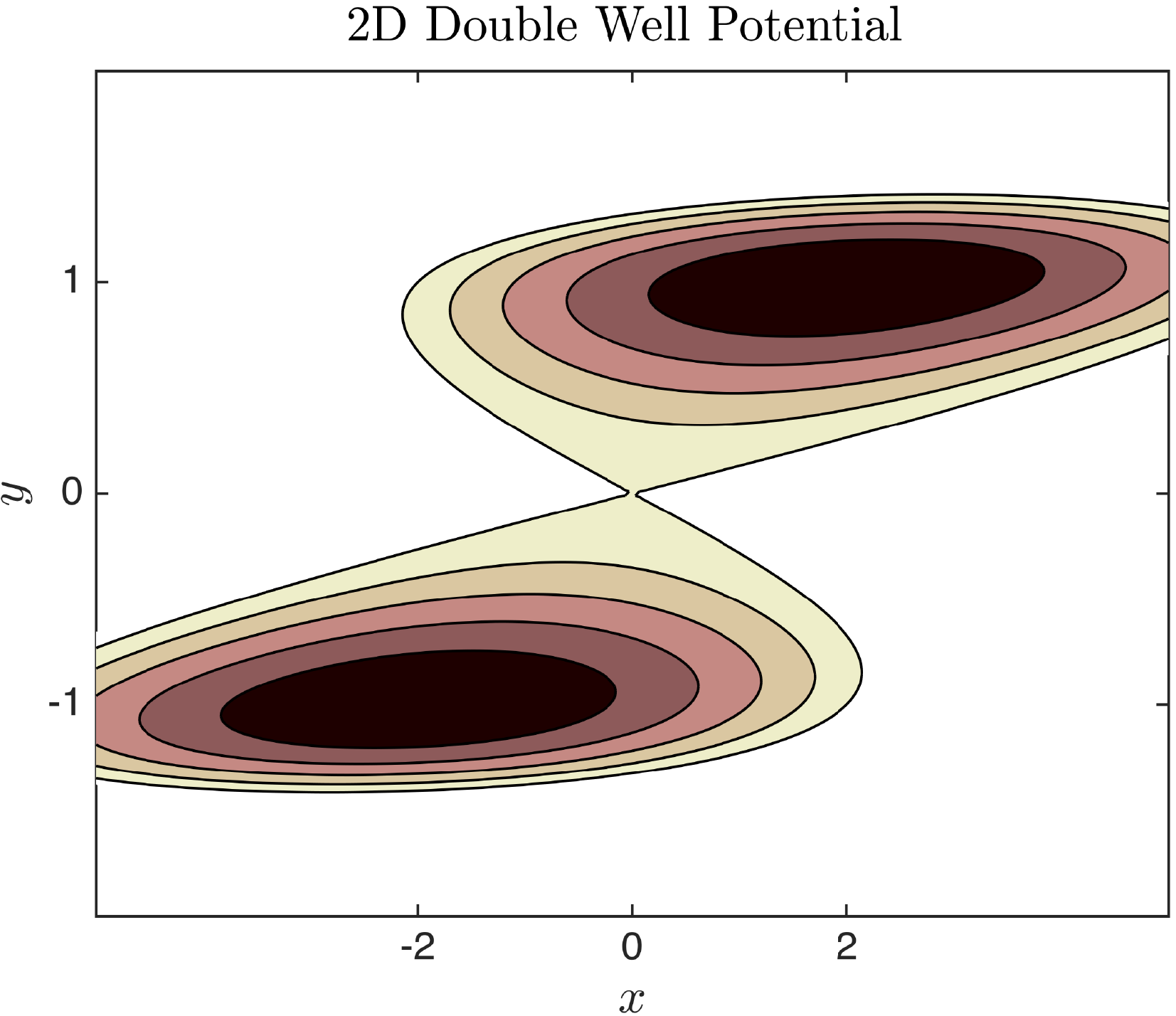}
\end{center}
\caption{ \small {\bf Contours of $\Phi(x)$ in \eqref{eq:dw_2d}.}  This potential energy function has local minima located at $x^{\pm} = (\pm 2, \pm 1)$ with $\Phi(x^+) = \Phi(x^-)$,
and a saddle point at the origin.  }
 \label{fig:dw_2d}
\end{figure}

\begin{figure}
\begin{center}
\includegraphics[width=0.45\textwidth]{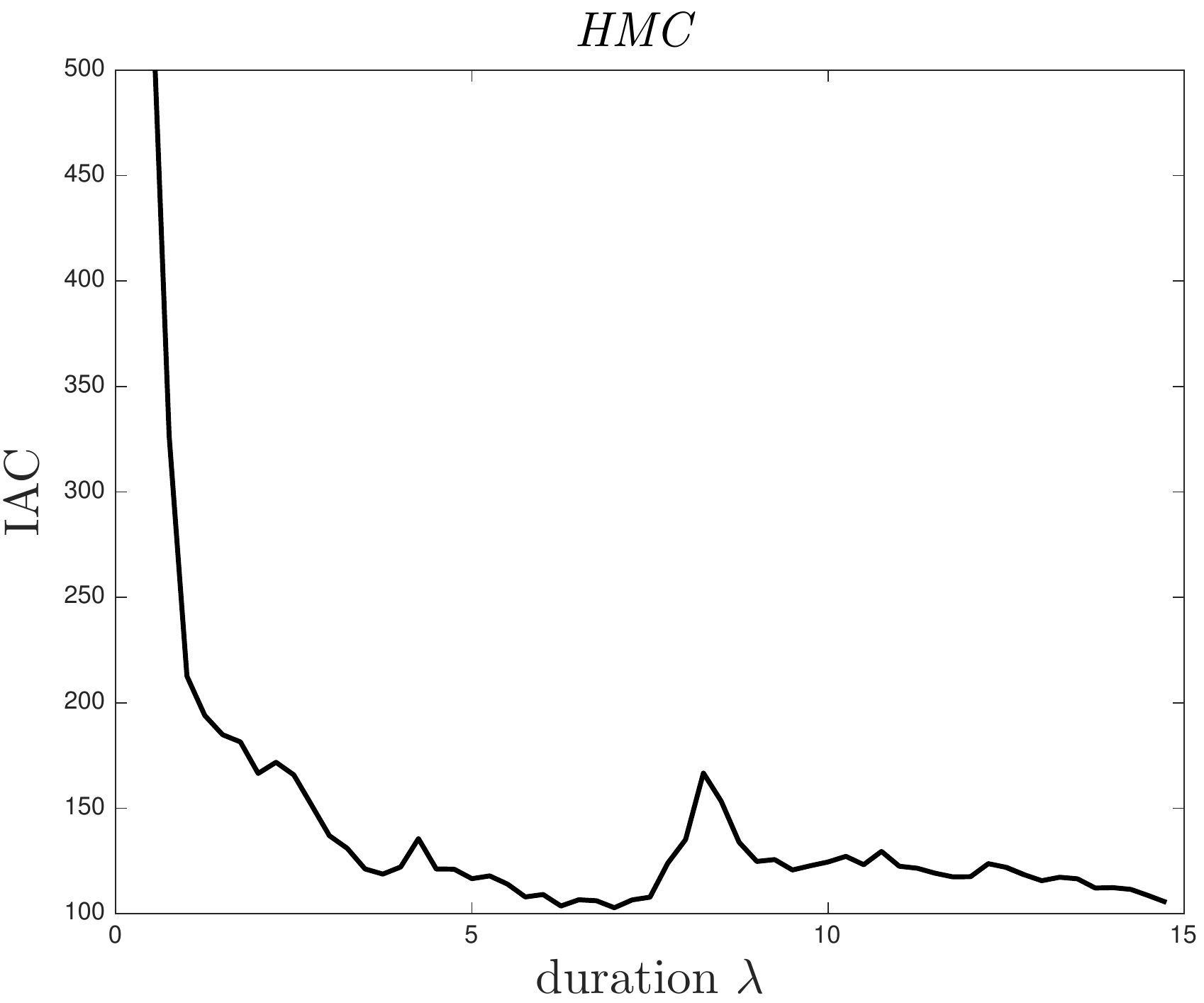}
\includegraphics[width=0.45\textwidth]{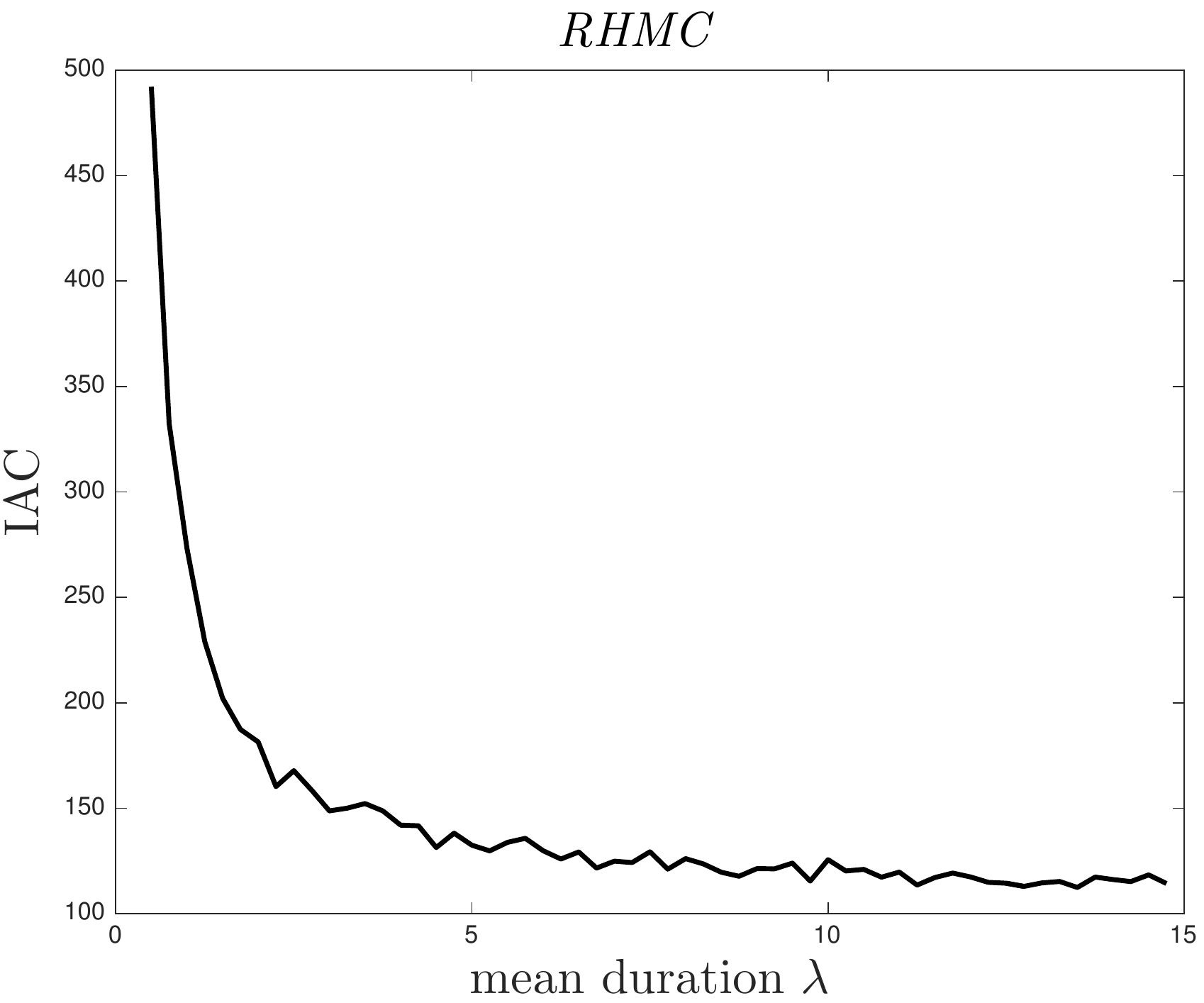}
\end{center}
\caption{ \small {\bf IAC for 2D Double Well.} The left (resp.~right) panel plot an estimate of the IAC associated to the time series $\{  2 Q_1(t_{i-1}) + Q_2(t_{i-1})  \}_{1 \le i \le 10^6}$ vs.~duration (resp.~mean duration) $\lambda$ for a HMC (resp.~RHMC) method applied to the potential energy function graphed in Figure~\ref{fig:dw_2d}.   }
 \label{fig:iac_dw_2d}
\end{figure}

\begin{figure}
\begin{center}
\includegraphics[width=0.45\textwidth]{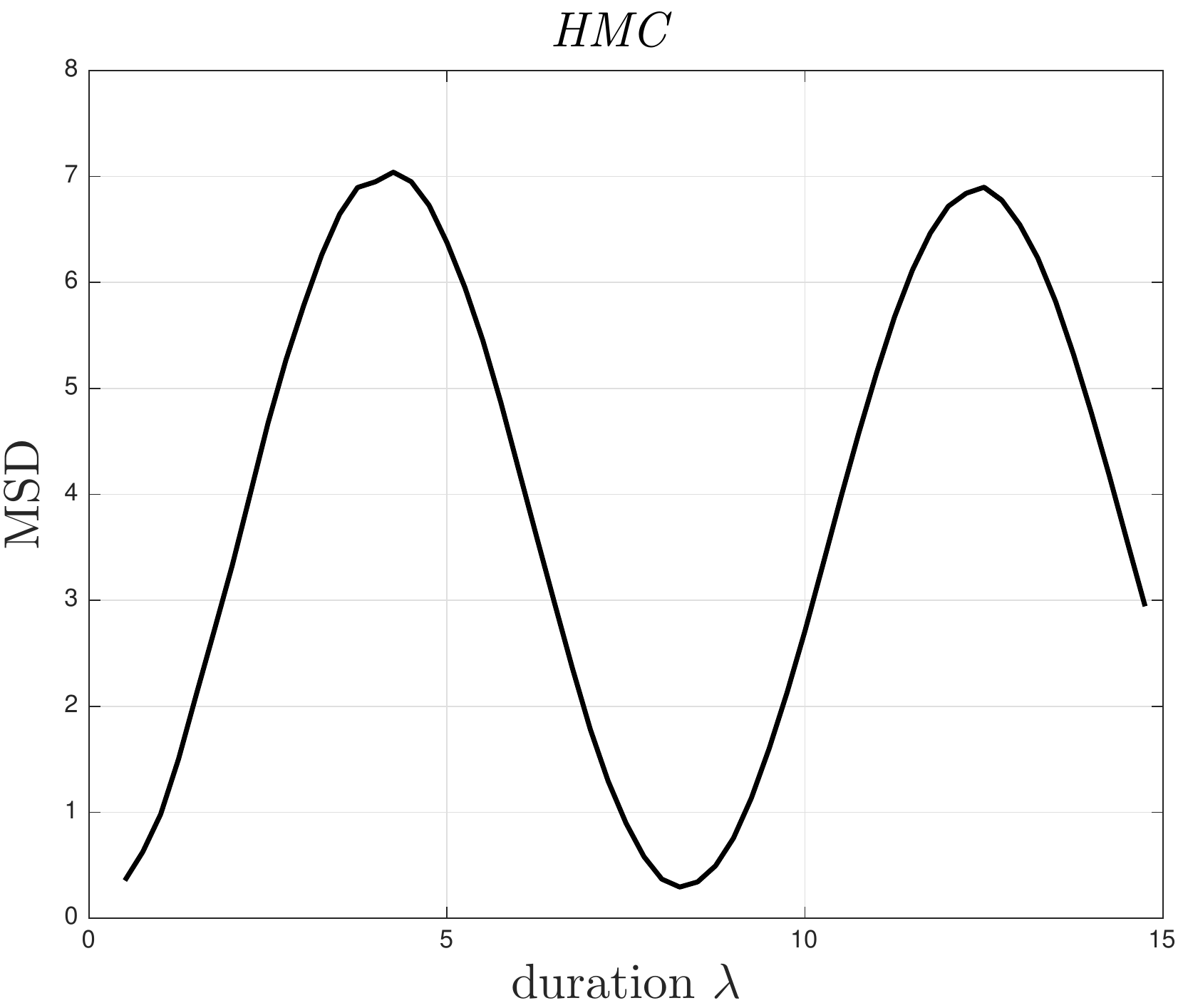}
\includegraphics[width=0.45\textwidth]{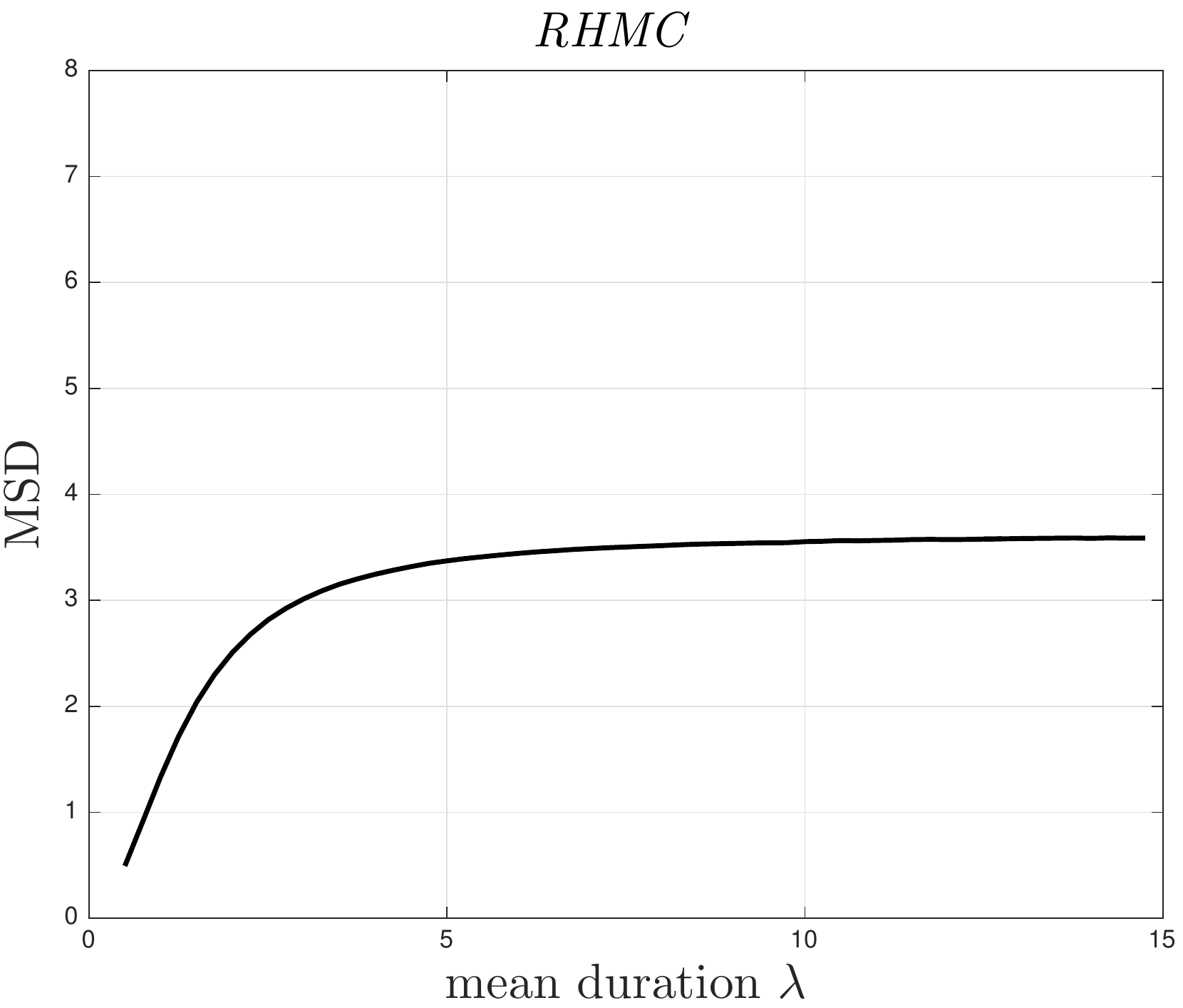}
\end{center}
\caption{ \small {\bf MSD for 2D Double Well.} The left (resp.~right) panel plot an estimate of the equilibrium MSD vs.~duration (resp.~mean duration) $\lambda$ for a HMC (resp.~RHMC) method applied to the potential energy function graphed in Figure~\ref{fig:dw_2d}. }
 \label{fig:msd_dw_2d}
\end{figure}


\subsection{Fifteen-dimensional Pentane Molecule}

Consider sampling from the equilibrium distribution of a chain of five particles with a potential energy that is a function of bond lengths, bond angles, dihedral angles, and inter-particles distances.
This potential energy is described in greater detail in \cite{CaLeSt2007}, which is based on the model described in \cite{MaSi1998}.
Figure~\ref{fig:pentane_conformations} illustrates and describes the three types of minima of the potential energy for the parameter values selected. Due to symmetries, each type corresponds to several  modes of the distribution.  For instance, there is  stable configuration where particles 1 and 5 are above the plane of particles 2, 3, 4 and another of the same type where they are below.
Figure~\ref{fig:iac_pentane_15d} (resp.~\ref{fig:msd_pentane_15d}) compares the IAC (resp.~MSD) of HMC and RHMC.
For this moderate dimensional distribution, we again observe a complex dependence of the performance of HMC on $\lambda$.


\begin{figure}
\begin{center}
\includegraphics[width=0.5\textwidth]{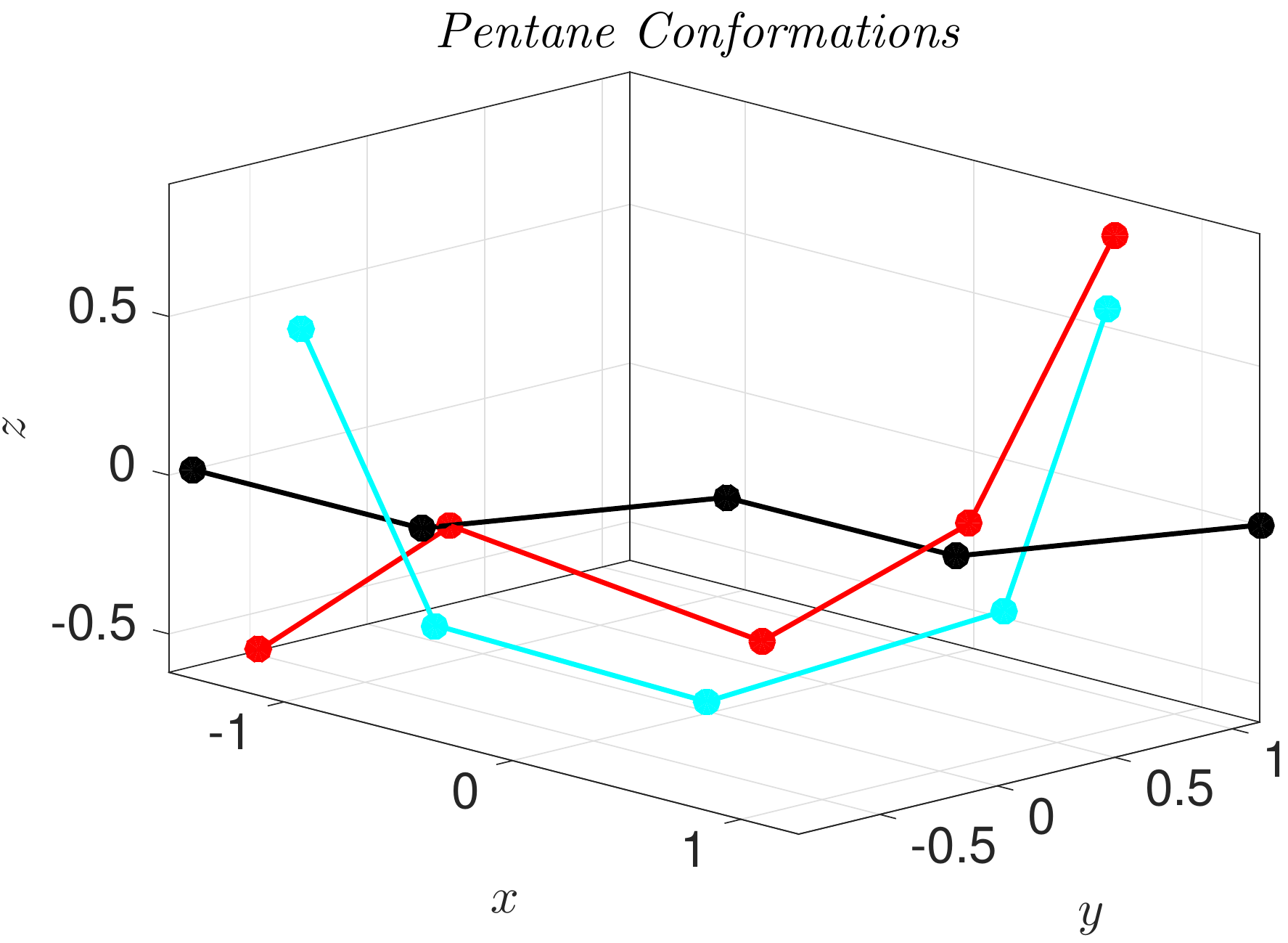}
\end{center}
\caption{ \small {\bf 15D Pentane Molecule.}  This figure shows the three types of potentially stable conformations of the pentane molecule.
}
 \label{fig:pentane_conformations}
\end{figure}

\begin{figure}
\begin{center}
\includegraphics[width=0.45\textwidth]{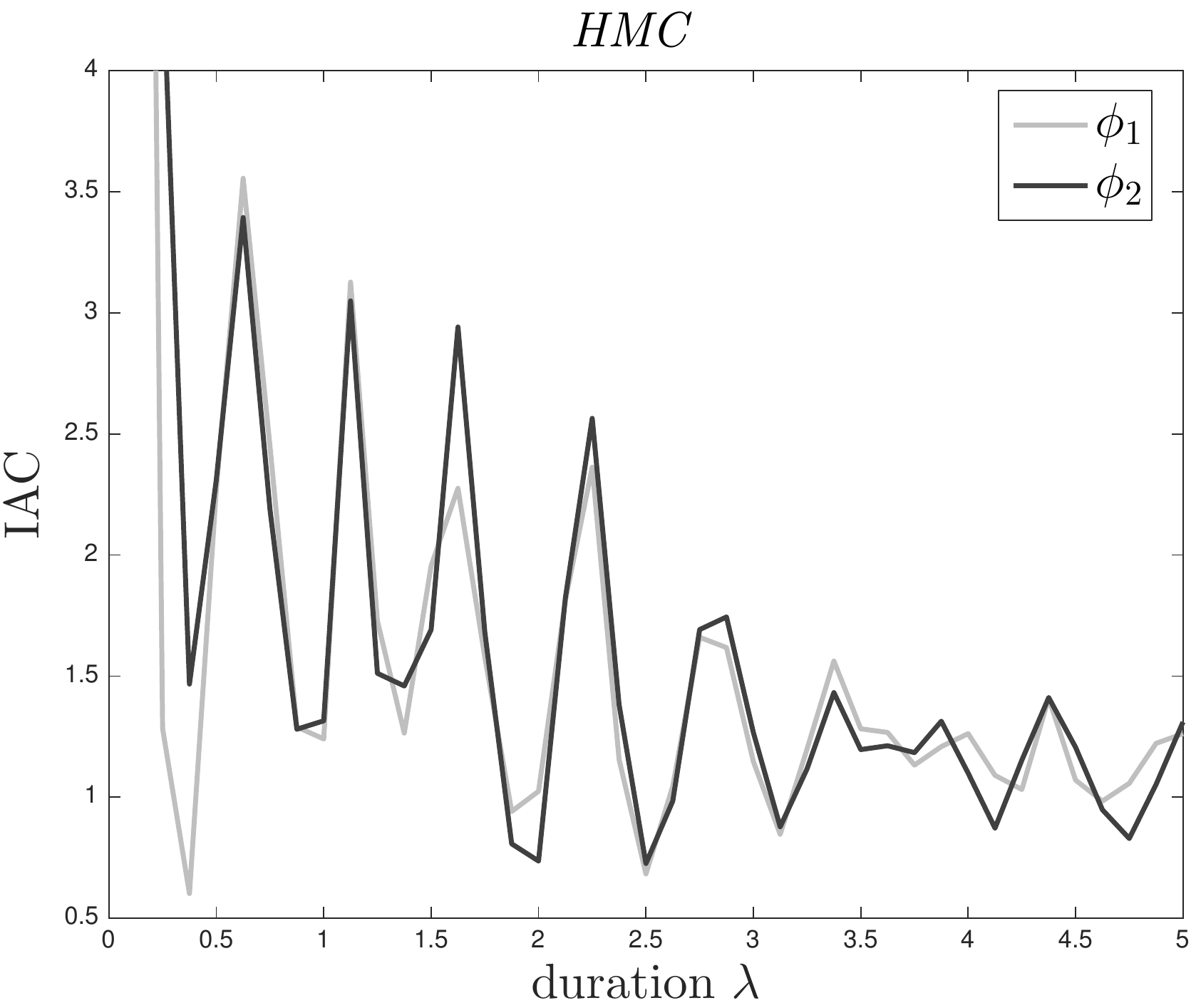}
\includegraphics[width=0.45\textwidth]{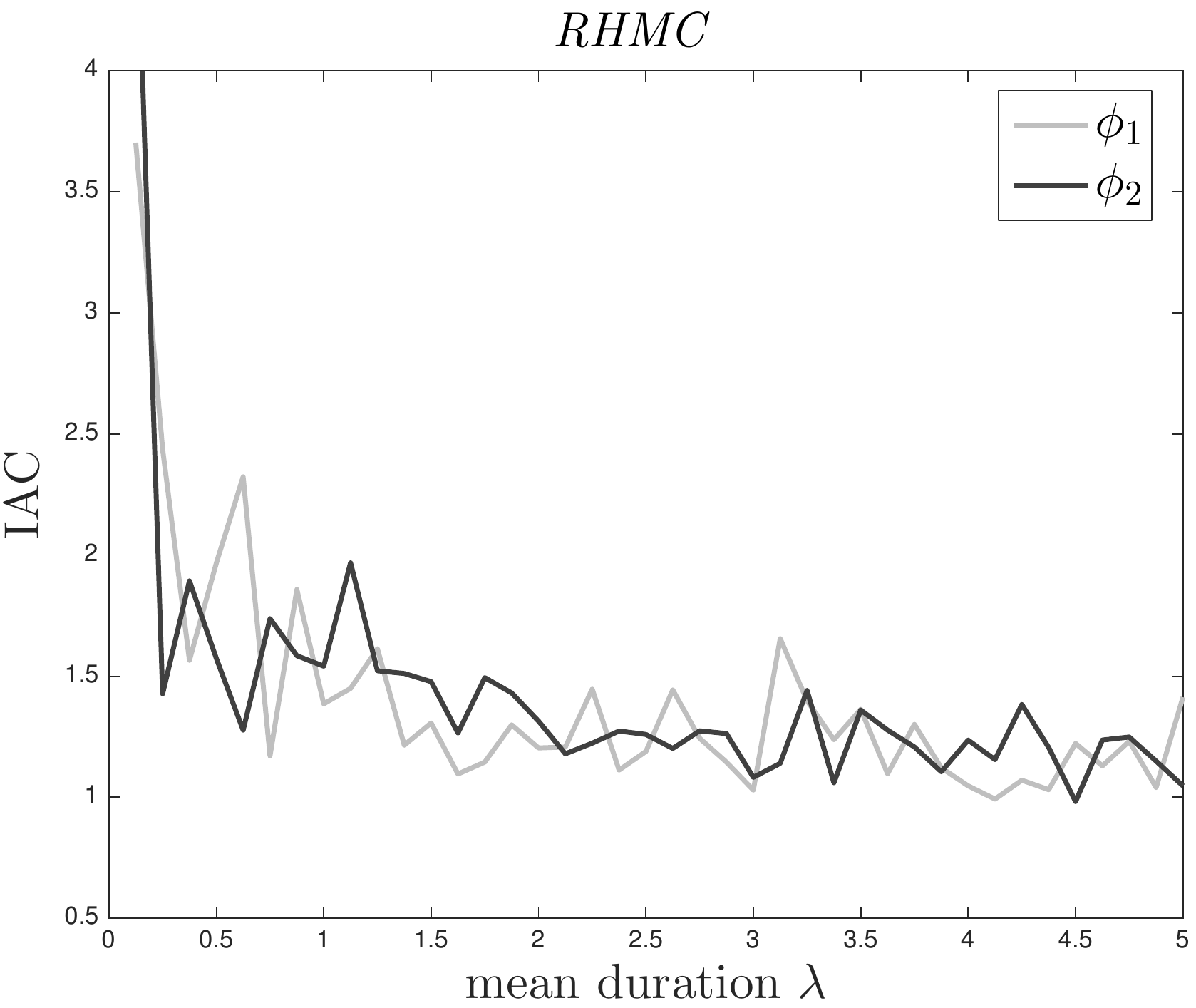}
\end{center}
\caption{ \small {\bf IAC for 15D Pentane.} The left (resp.~right) panel plot an estimate of the IAC associated to the dihedral angles (labelled $\phi_1$ and $\phi_2$ in the figure legends) vs.~duration (resp.~mean duration) $\lambda$ for a HMC (resp.~RHMC) method applied to the potential energy function of the pentane molecule  described in the text.
The time series each consist of $10^4$ samples.
}
 \label{fig:iac_pentane_15d}
\end{figure}

\begin{figure}
\begin{center}
\includegraphics[width=0.45\textwidth]{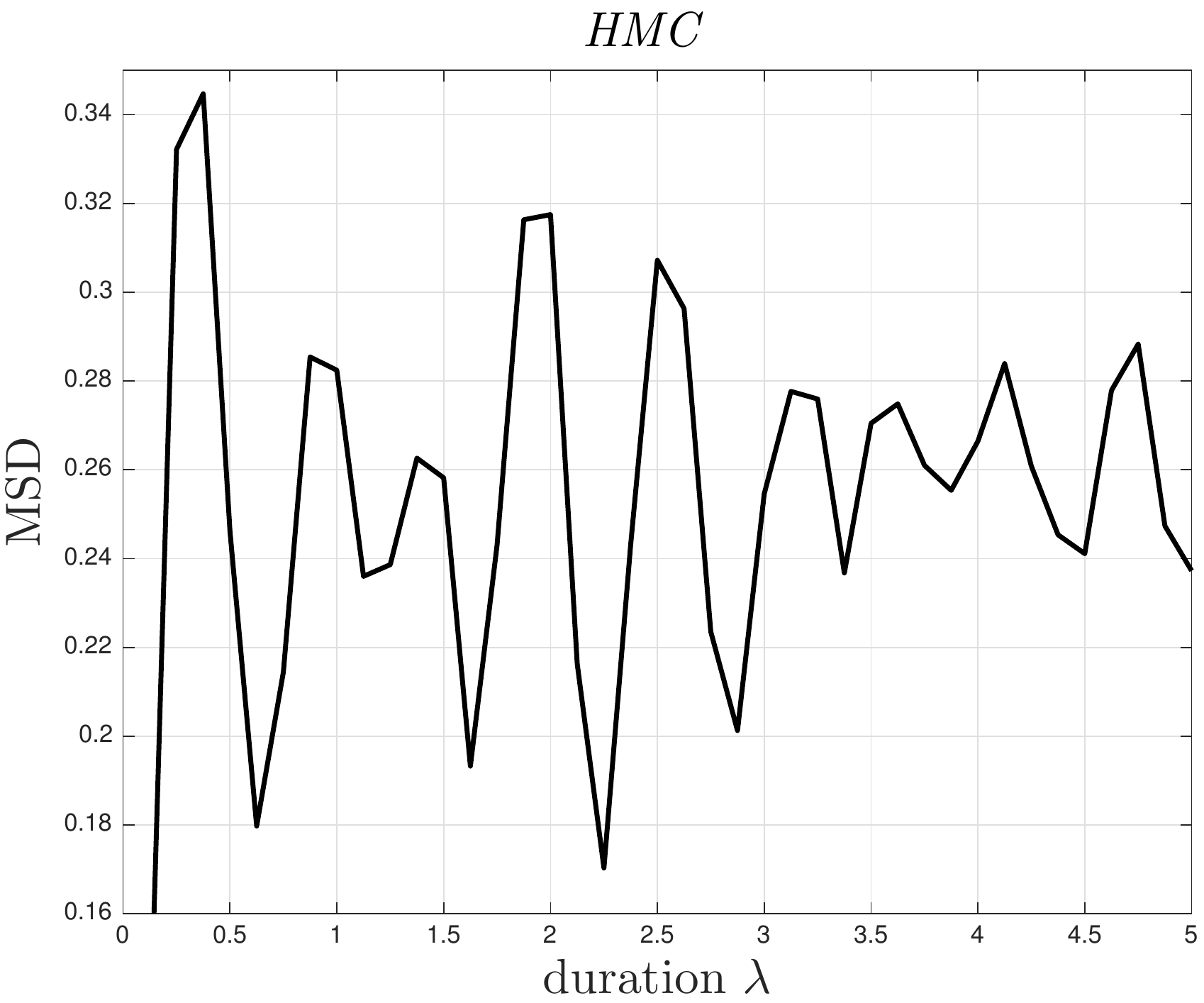}
\includegraphics[width=0.45\textwidth]{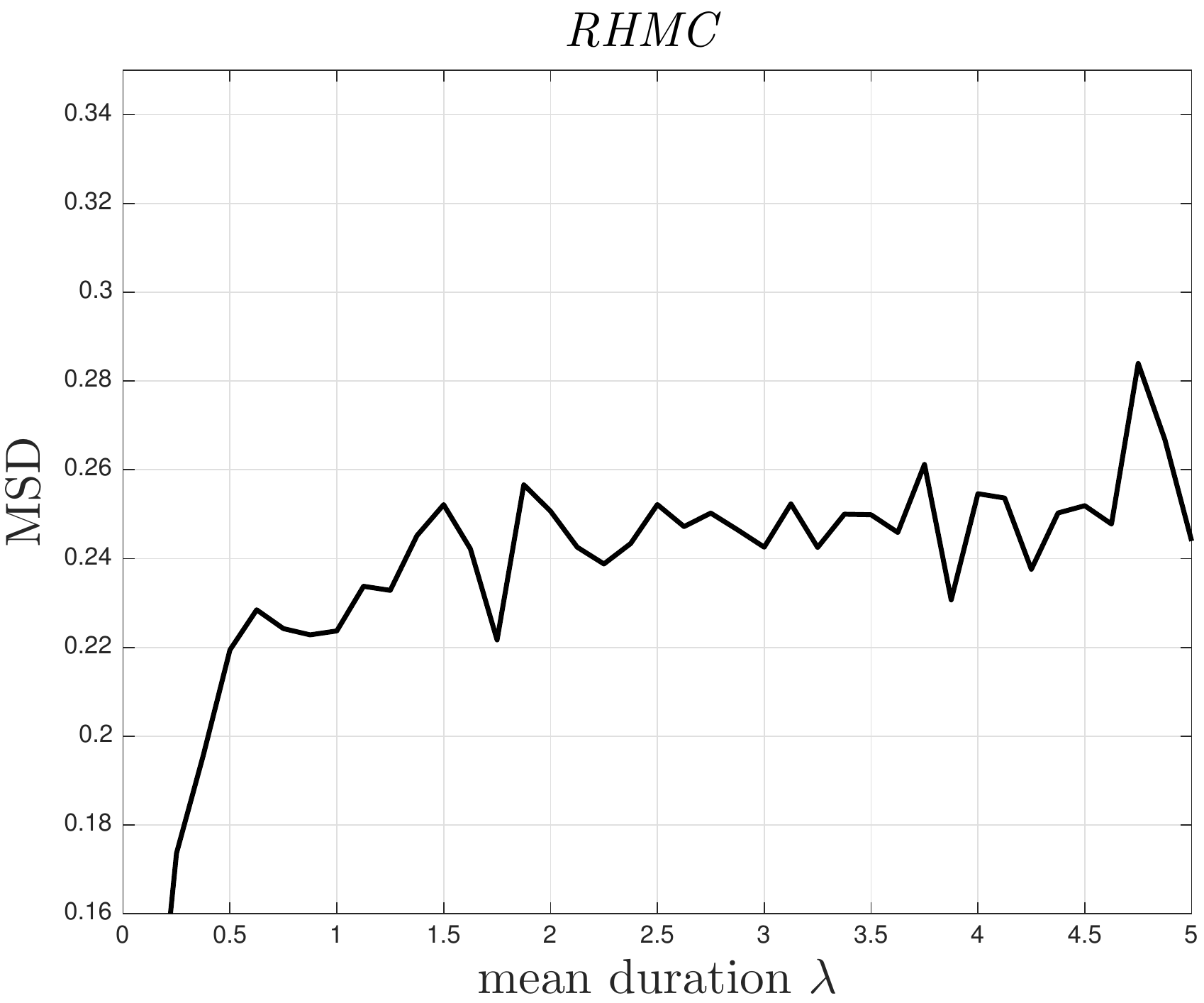}
\end{center}
\caption{ \small {\bf MSD for 15D Pentane.} The left (resp.~right) panel plot an estimate of the equilibrium MSD vs.~duration (resp.~mean duration) $\lambda$ for a HMC (resp.~RHMC) method applied to the potential energy function of the pentane molecule  described in the text.
The time series each consist of $10^4$ samples.
 }
 \label{fig:msd_pentane_15d}
\end{figure}


\clearpage

\section{Outlook} \label{sec:outlook}

All of the preceding developments have taken place in the exact integration scenario. The most obvious modification of Algorithm \ref{algo:rhmc} that takes into account integration errors approximates the Hamiltonian flow by a volume-preserving, reversible integrator (such as Verlet) with time step $\Delta t$, updates time via $t_1= t_0+M \Delta t$, where  $M$ is a random number of time-steps that is geometrically distributed with mean $\lambda/\Delta t$, and uses an accept-reject mechanism to remove the bias due to the energy errors introduced by this integrator. The study of that algorithm involves many numerical analysis technicalities, as it is necessary to derive global error bounds for the numerical solution that are valid over integration legs of arbitrarily long length $M\Delta t$. Such developments do not have much relation with the mathematical techniques used here so far and are out of the scope of the present article.

In this section we  suggest two variants of Algorithm \ref{algo:rhmc} that do not use the exact solution of the Hamiltonian dynamics. These variants are based on approximating the gradients in \eqref{eq:generator} with the help of a numerical integrator and are therefore similar to the recently introduced generalized Markov Chain Approximation Methods (MCAM)  \cite{KuDu2001, BoVa2016}. We shall not be concerned with  comparing  the efficiency  of the modifications of Algorithm \ref{algo:rhmc} introduced here with that of standard HMC and related techniques.

In what follows, for a given numerical integrator, we denote by $\theta_{\Delta t}$ the map that advances the solution of Hamilton's equations over a single time-step of length $\Delta t$; a typical example is provided by the (velocity) Verlet integrator $(q_1,p_1) = \theta_{\Delta t}(q,p)$:
\[
\begin{pmatrix} q_1 \\ 
p_1
\end{pmatrix} = \begin{pmatrix} q + \Delta t p - \dfrac{(\Delta t)^2}{2} \nabla \Phi(q) \\ 
p - \dfrac{\Delta t}{2} ( \nabla \Phi(q) + \nabla \Phi(q_1) ) \end{pmatrix} \;.
\]

\subsection{Variant \#1} This variant is defined by the following approximation to the generator \eqref{eq:generator}  of RHMC
\[
 L_h^{(1)}f(q,p) = \lambda^{-1} \E \left\{ f(  \Gamma (q,p) ) - f(q,p) \right\}  + h^{-1}\left( f(\theta_h(q,p)) - f(q,p) \right),
\]
where $h>0$ is a parameter; since the Liouville operator acting on a test function $f$ represents the time-derivative of $f$ along the solutions of Hamilton's equations, it is clear that, for consistent integrators and smooth test functions,
\[
 L_h^{(1)} f(q,p) = L f(q,p) + O(h) \;.
\]
The random jump times and the embedded chain of the Markov process generated by $ L_h^{(1)}$ can be produced by iterating the following algorithm.

\begin{algorithm} \label{algo:approx_1}
Given a duration parameter $\lambda > 0$, a step length parameter $h>0$, the current time $t_0$, and the current state $( \tilde Q(t_0), \tilde P(t_0) ) \in \mathbb{R}^{2D}$,  output an updated state  $( \tilde Q(t_1), \tilde P(t_1) )$ at the random time $t_1$ using two steps
\begin{description}
\item[Step 1] Generate an exponential random variable $\delta t$ with mean $h \lambda / (h + \lambda)$, and update time via $t_1 = t_0 + \delta t$.
\item[Step 2] Generate a uniform random variable $u \sim U(0,1)$ and set \[
( \tilde Q(t_1), \tilde P(t_1) ) = \begin{dcases}
\Gamma(\tilde Q(t_0), \tilde P(t_0) ) & u \le \frac{h}{h + \lambda} \\
\theta_h(\tilde Q(t_0), \tilde P(t_0) )  & \text{otherwise}
\end{dcases}
\]
where $\Gamma$ is the momentum randomization map given in \eqref{eq:momentum_randomization}.
\end{description}
\end{algorithm}

 The random variables $\delta t$, $u$, and the random vector $\xi$ in the momentum randomization map $\Gamma$ are independent. In terms of the sequence of random jump states $\{ (\tilde Q(t_i), \tilde P(t_i) ) \}$ and random jump times $\{ t_i \}$,
 the trajectory of the process is:
 \[
(Q(t), P(t)) = ( \tilde Q(t_i), \tilde P(t_i) )  \quad \text{for  $t \in [t_i, t_{i+1})$} \;.
\]
For any $t>0$, the time-average of a function $f: \mathbb{R}^{2D} \to \mathbb{R}$ along this trajectory is  given by: \[
\int_0^t f(Q(s), P(s)) ds = \sum_{0 \le i \le \infty} f(Q(t_i), P(t_i) ) (t \wedge t_{i+1} - t \wedge t_i )\;.
\]
The mean holding time of this Markov jump process is constant and given by \begin{equation}\label{eq:ht}
\E \delta t = \frac{h \lambda}{h + \lambda} \;.
\end{equation}  If $\lambda$ is large and $h$ is small, this process mainly jumps from $(q,p)$ to $\theta_h(q,p)$, with occasional randomizations of momentum; in other words, the jump states come from integration legs of the Hamiltonian dynamics interspersed with occasional momentum randomizations. Note that, while the  holding times $\delta t$ are random, the step-length parameter  in the numerical integrator remains constant.
Figure~\ref{fig:mcam_rhmc_gaussian_1d} illustrates the use of the algorithm  in the case a one-dimensional standard normal target distribution.

\begin{figure}[ht!]
\begin{center}
\includegraphics[width=0.45\textwidth]{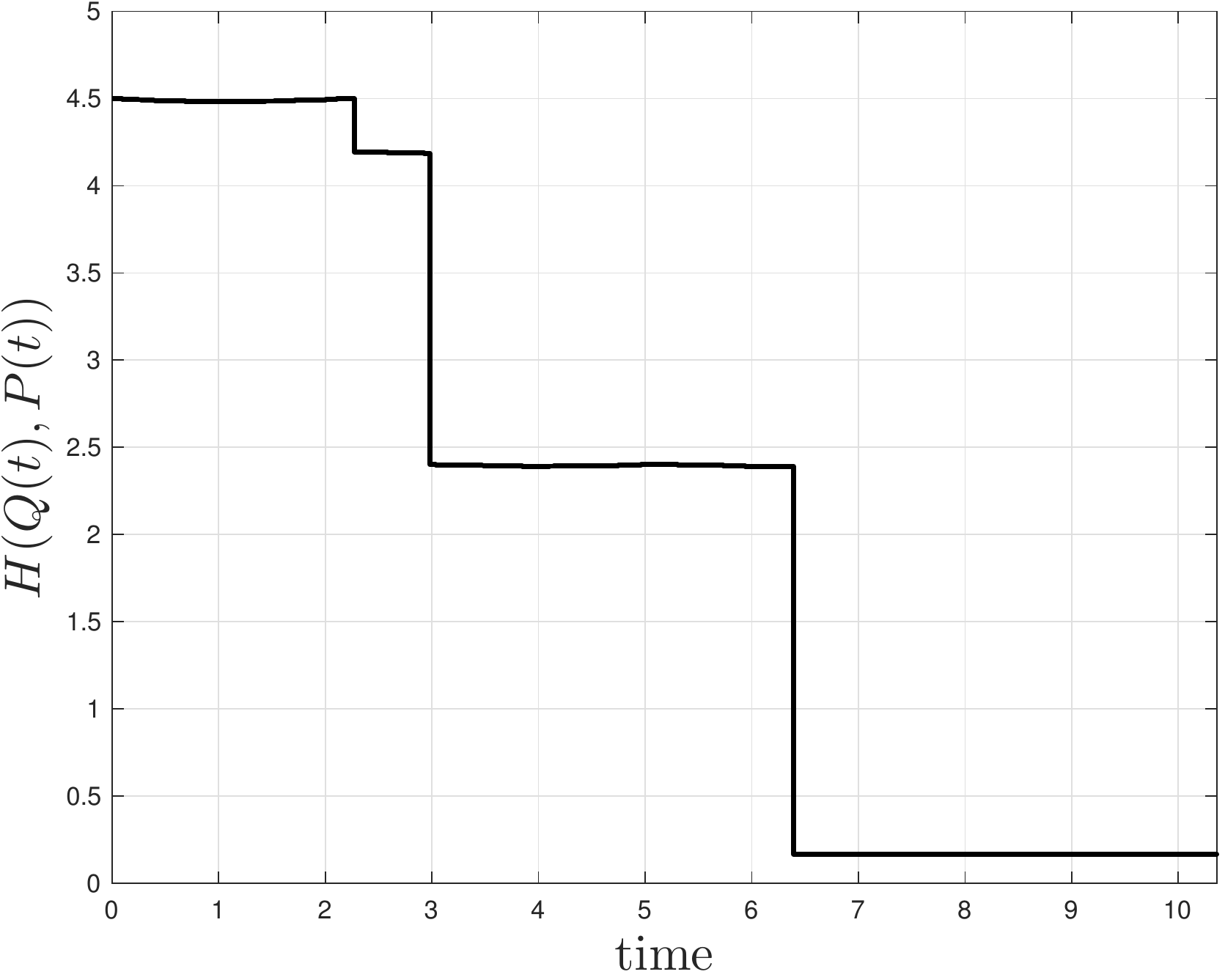}
\includegraphics[width=0.45\textwidth]{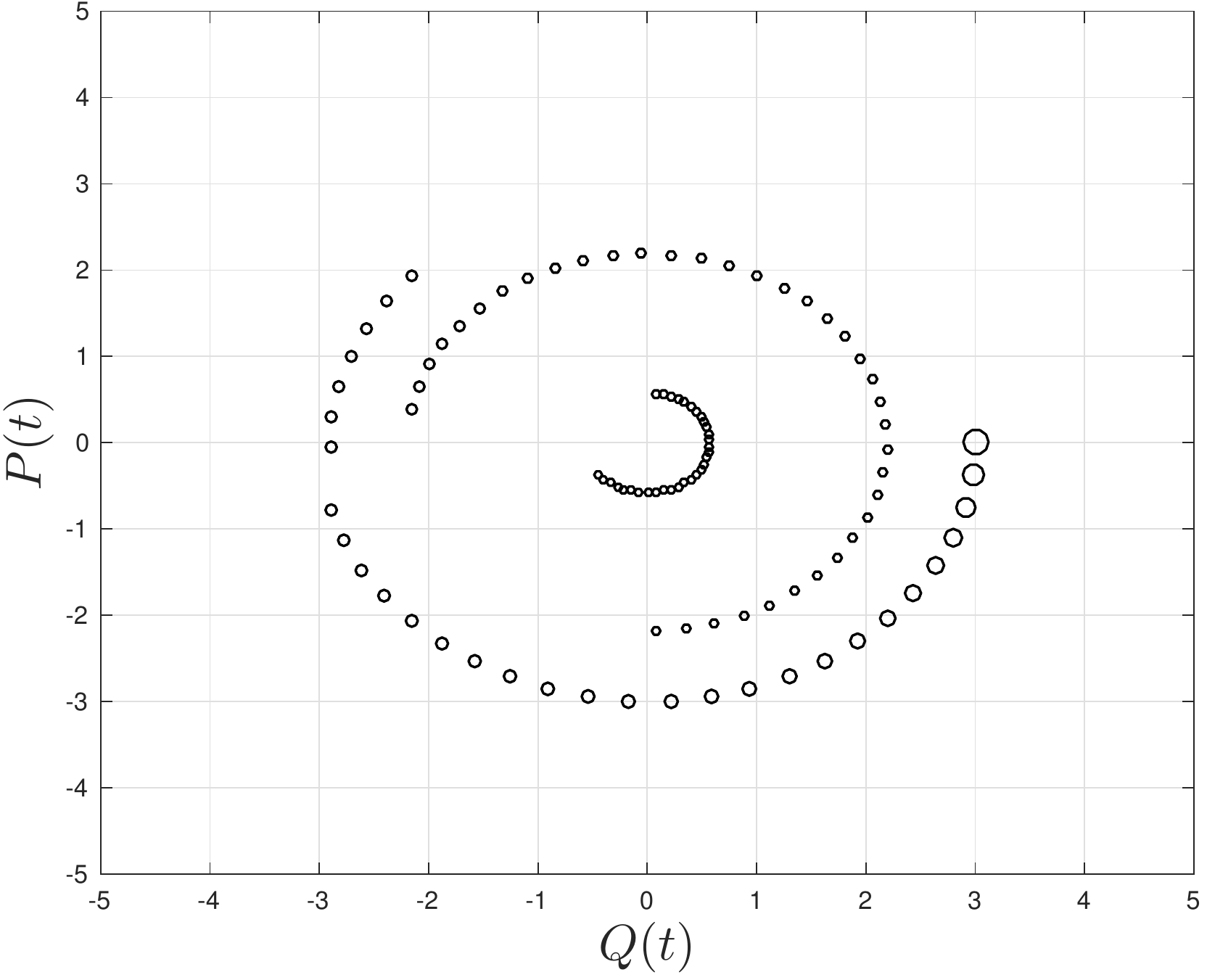} \\
\includegraphics[width=0.45\textwidth]{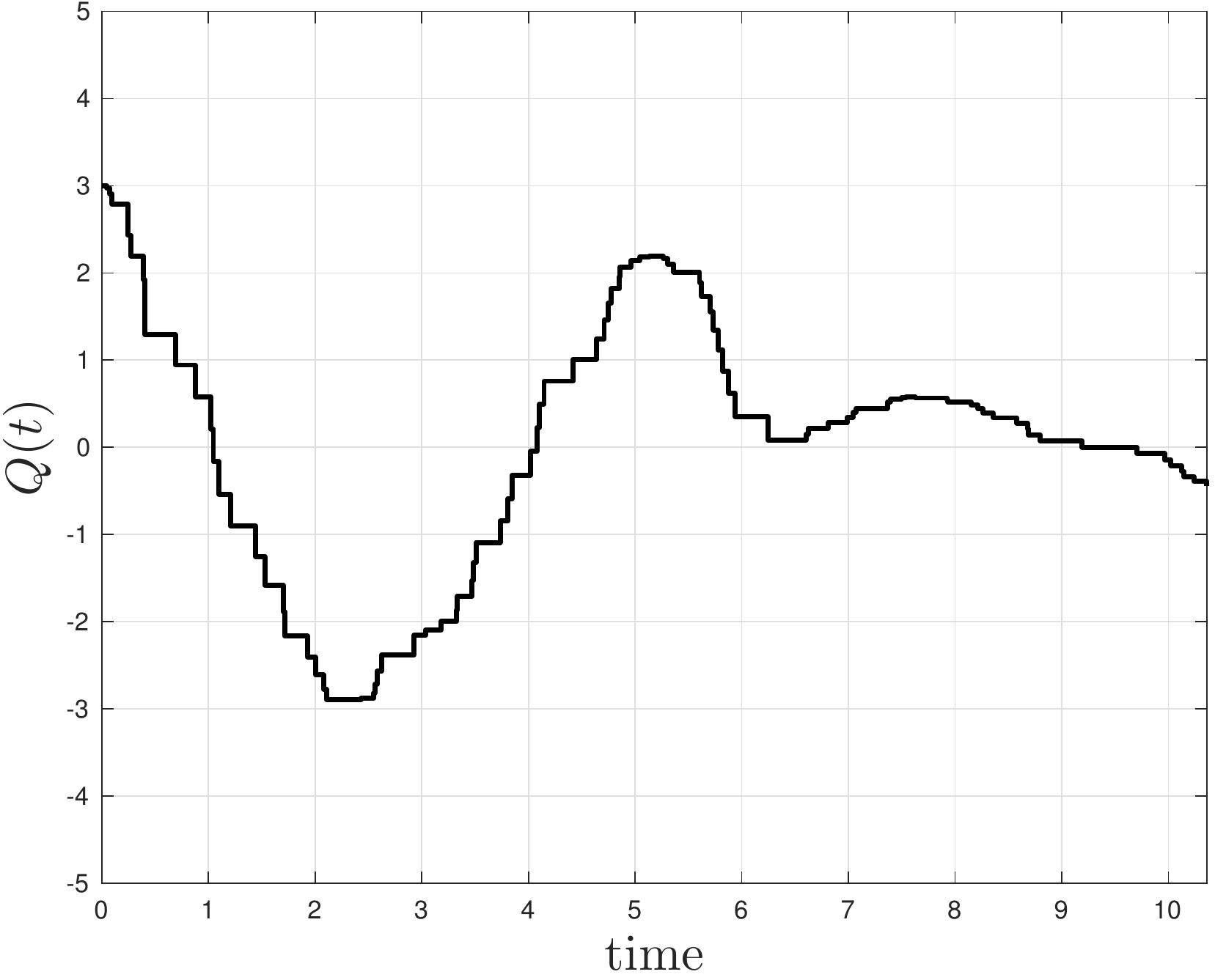}
\includegraphics[width=0.45\textwidth]{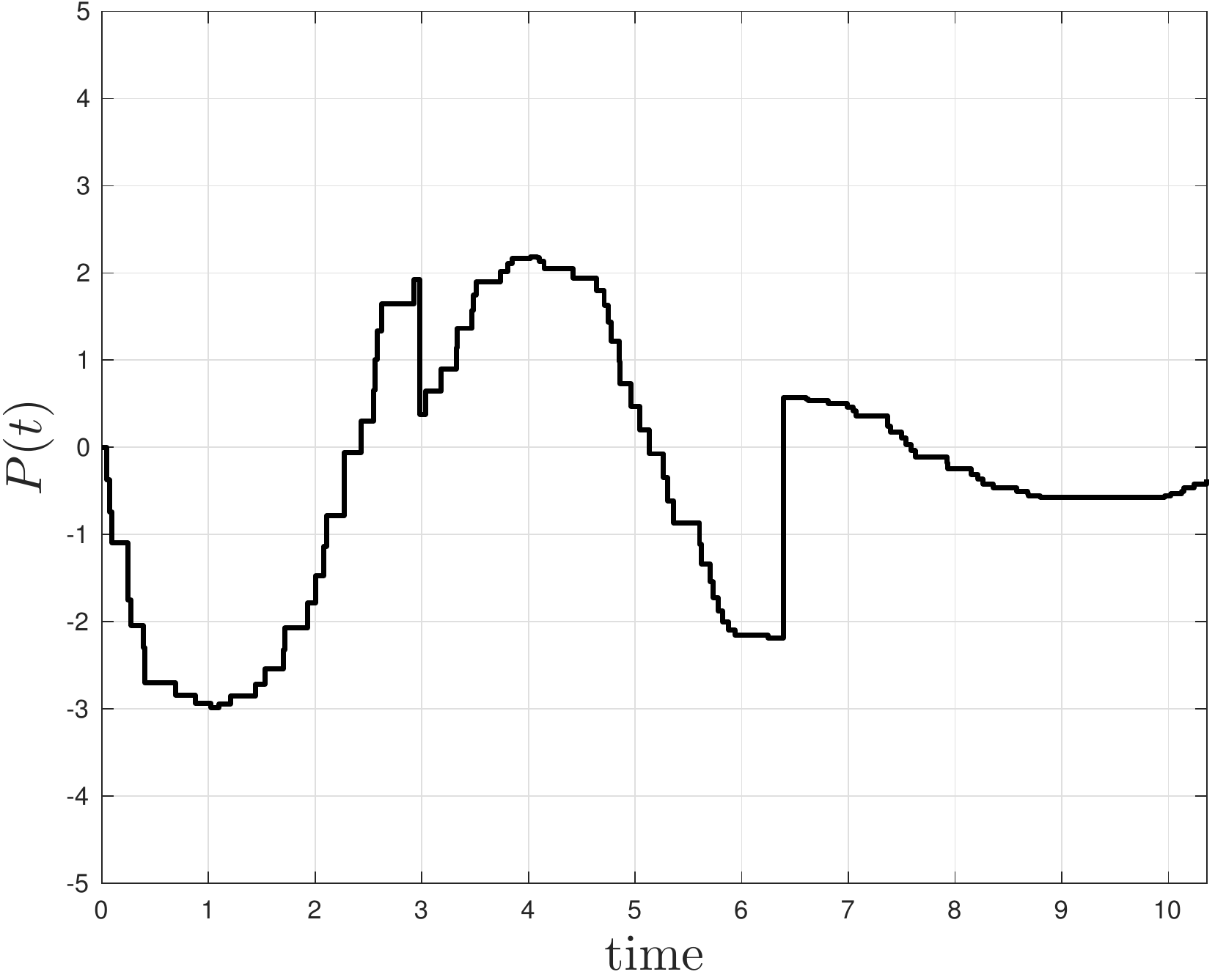}
\end{center}
\caption{ \small {\bf Sample Paths of Variant \#1.}
These figures show a realization produced by iterating Algorithm~\ref{algo:approx_1} in the case of a one-dimensional normal target distribution. The  step length is $h=0.125$, and the duration $\lambda = \pi$ (these values are chosen for visualization purposes only).  The top left panel shows the evolution of the total energy as a function of time.  Note
that there have been three momentum randomizations and that, between them,
 the total energy is essentially constant.  The top right panel shows the evolution of the discrete trajectory in phase space.  The size of the markers is related to $t$: points along the trajectory corresponding to larger values of $t$  have smaller markers.  The bottom panels show the position and momentum as functions of time.  The holding time in each state is random, and $Q(t)$ and $P(t)$ are piecewise constant in time.  
 }
 \label{fig:mcam_rhmc_gaussian_1d}
\end{figure}

\clearpage

 We now show that the generator $L_h^{(1)}$ inherits the stochastic Lyapunov function of $L$. The hypotheses on the potential energy that we use imply that the tails of the target
 $\exp(-\Phi)$ are no lighter or heavier than those of a Gaussian distribution; it is likely that those hypotheses may be relaxed. The hypotheses on $\theta_h$ are satisfied by any reasonable integrator such as the Euler rule or the Verlet method.

\begin{prop} \label{prop:tilde_Lh}
Suppose that, in addition to  Hypothesis~\ref{hypothesis_on_phi}, the potential energy is twice-differentiable and satisfies the following conditions.
\begin{enumerate}
\item There exists $C>0$ such that for all $q \in \mathbb{R}^D$: \[
 \| D^2 \Phi(q) \| \le C \;.
 \]
\item There exists $C>0$ such that for all $q \in \mathbb{R}^D$: \[
\Phi(q) \ge C (1 + |q|^2 ) \;.
 \]
\end{enumerate}
Furthermore assume that the integrator is such that,
for any
 $f \in C^2(\mathbb{R}^{2D}, \mathbb{R})$ with the property that there exists $C_1>0$ such that \[
 \| D^2 f (q,p) \|  \le C_1 \quad \forall (q,p) \in \mathbb{R}^{2D} \;,
\]   there exists a  constant $C_2(f)>0$ such that  \[
f(\theta_h(q,p)) =  f(q,p) + h \mathcal{L} f(q,p) + R_f(q,p,h) \;,
\] where  \[
|R_f(q,p,h) | \le C_2(f) (1 + H(q,p))  h^2 \quad \forall (q,p) \in \mathbb{R}^{2D} \;.
\]
Then, for $h$ sufficiently small, there exist $\gamma_h>0$ and $K_h \ge 0$ such that: 
\[
L_h^{(1)} V(q,p) \le - \gamma_h V(q,p) + K_h \quad \forall (q,p) \in \mathbb{R}^{2D} \;.
\]
\end{prop}

\begin{proof}
By hypothesis, there exists a constant $C_V>0$ such that \[
h^{-1} (V(\theta_h(q,p)) - V(q,p)) \le \mathcal{L} V(q,p) + h  (1 + V(q,p) ) C_V.
\] Since $V$ is a Lyapunov function for $L$, we have that \begin{align*}
L_h^{(1)} V(q,p) &\le L V(q,p)  + h  (1 + V(q,p) ) C_V  \\
&\le - \gamma V(q,p) + K  + h  (1 + V(q,p) ) C_V \\
&\le - (\gamma - h C_V) V(q,p) + ( K + h C_V)
\end{align*}
which gives the desired result with $\gamma_h = \gamma - h C_V$ and $K_h = K + h C_V$.
\end{proof}


Unfortunately, due to the discretization error, the integrator $\theta_h$ does not  preserve the Hamiltonian function exactly, and thus, the invariant measure of $L_h^{(1)}$ is not the invariant measure of $L$.

\subsection{Variant \#2}

In order to correct the bias in $L_h^{(1)}$, we add to Algorithm~\ref{algo:approx_1}  the possibility of  additional jumps (momentum flips) from $(q,p)$ to $\phi(q,p) = (q,-p)$, as follows:

\begin{algorithm}\label{algo:approx_2}
 Let $\nu(q,p) = \exp( - H(q,p))$ and define the Metropolis ratio \[
 \alpha_h(q,p) = 1 \wedge \frac{\nu(\theta_h(q,p))}{\nu(q,p)}
 \]
 for all $(q,p) \in \mathbb{R}^{2D}$.   Given a duration parameter $\lambda > 0 $, a step length parameter $h>0$, the current time $t_0$, and the current state $( \tilde Q(t_0), \tilde P(t_0) ) \in \mathbb{R}^{2D}$, the method outputs an updated state  $( \tilde Q(t_1), \tilde P(t_1) )$ at the random time $t_1$ using two steps
\begin{description}
\item[Step 1] Generate an exponential random variable $\delta t$ with mean $h \lambda / (h + \lambda)$, and update time via $t_1 = t_0 + \delta t$.
\item[Step 2] Generate a uniform random variable $u \sim U(0,1)$ and set \begin{align*}
& ( \tilde Q(t_1), \tilde P(t_1) ) = \\
& \quad \begin{dcases}
\Gamma(\tilde Q(t_0), \tilde P(t_0) ) & u\leq \frac{h}{h+\lambda}\\
\theta_h(\tilde Q(t_0), \tilde P(t_0) ) & \frac{h}{h+\lambda} < u \le \frac{h+\alpha_h(\tilde Q(t_0), \tilde P(t_0) )\lambda }{h+\lambda}\\
\phi(\tilde Q(t_0), \tilde P(t_0) ) &  \text{otherwise}
\end{dcases}
\end{align*}
where $\Gamma$ is the momentum randomization map given in \eqref{eq:momentum_randomization}.
\end{description}
\end{algorithm}

The infinitesimal generator  associated to Algorithm~\ref{algo:approx_2}  is given by 
\begin{equation}\label{mcam_rhmc}
\begin{aligned}
&L_h^{(2)}f(q,p) = \lambda^{-1} \E \left\{ f(  \Gamma (q,p) ) - f(q,p) \right\}\\&  \qquad \qquad+ h^{-1}\alpha_h(q,p) \left( f(\theta_h(q,p)) - f(q,p) \right)  \\
& \qquad \qquad+ h^{-1} \left( 1 -  \alpha_h(q,p)\right) \left( f(\phi(q,p)) - f(q,p) \right)\;.
\end{aligned}
\end{equation}
 Like $L_h^{(1)}$, the   generator $L_h^{(2)}$ induces a Markov jump process with a constant mean holding time given by \eqref{eq:ht}. If $\lambda$ is large and $h$ is small, this process mainly jumps from $(q,p)$ to $\theta_h(q,p)$, with occasional randomizations of momentum and momentum flips. As we show next, the weights $\alpha_h$ and $(1-\alpha_h)$ in \eqref{mcam_rhmc} have been chosen to ensure  that, for suitable integrators, $L_h^{(2)}$ has the same infinitesimally invariant measure of $L$.
 
\begin{prop}
Suppose that the integrator $\theta_h$  is {\em reversible}, i.e. $\theta_{\Delta t}\circ\phi\circ \theta_{\Delta t} = \phi$ and {\em volume-preserving}, i.e.\ $\det( D \theta_h) = 1$.
Then, for all $f \in C^{\infty}_c( \mathbb{R}^{2D} ) $, we have that \[
\int_{\mathbb{R}^{2D} } L_h^{(2)} f(q,p) \nu(q,p) dq dp = 0 \;.
\]
\end{prop}

\begin{proof}
Note that \begin{align*}
& \int_{\mathbb{R}^{2D} } L_h^{(2)} f(q,p) \nu(q,p) dq dp = \int_{\mathbb{R}^{2D} } \lambda^{-1} \E \left\{ f(  \Gamma (q,p) ) - f(q,p) \right\}  \nu(q,p) dq dp \\
&\quad +  \frac{1}{h} \int_{\mathbb{R}^{2D} }  \left( \nu(q,p) \wedge \nu(\theta_h(q,p))  \right) \left( f(\theta_h(q,p)) - f(q,p) \right) dq dp \\
&\quad + \frac{1}{h} \int_{\mathbb{R}^{2D} }  \left( \nu(q,p) -  \nu(q,p) \wedge \nu(\theta_h(q,p))  \right) \left( f(\phi(q,p)) - f(q,p) \right) dq dp
\end{align*}
As discussed in Proposition~\ref{prop:infinitesimal_invariance}, the first integral on the right-hand-side vanishes.  For the next two integrals, since $\theta_h$ is volume-preserving and reversible by hypothesis, a change of variables implies that
\begin{align*}
& \int_{\mathbb{R}^{2D} } L_h^{(2)} f(q,p) \nu(q,p) dq dp = \\
& \qquad \frac{1}{h} \int_{\mathbb{R}^{2D} }  \left[ \nu(\theta_h^{-1}(q,p)) \wedge \nu(q,p)  \right]  f(q,p)  dq dp  \\
& \quad -  \frac{1}{h} \int_{\mathbb{R}^{2D} }  \left[ \nu(q,p) \wedge \nu(\theta_h(q,p) )  \right]  f(q,p)  dq dp  \\
&\quad - \frac{1}{h} \int_{\mathbb{R}^{2D} }  \left[ \nu(\phi(q,p)) \wedge \nu(\theta_h \circ \phi(q,p) )  \right] f(q,p) dq dp \\
&\quad + \frac{1}{h} \int_{\mathbb{R}^{2D} }  \left[ \nu(q,p) \wedge \nu(\theta_h(q,p))  \right] f(q,p) dq dp  = 0
\end{align*}
where in the last step we used the fact that $\nu \circ \theta_h = \nu\circ \phi \circ \theta_h = \nu \circ \theta_h^{-1}$.
\end{proof}

We conjecture that Algorithm~\ref{algo:approx_2} is geometrically ergodic under suitable assumptions on the potential energy function and suitable choices of the reversible, volume preserving integrator. However, the proof of this appears to be involved because one needs to estimate carefully the behavior of the jump rate from $(q,p)$ to $\theta_h(q,p)$ in regions where $H$ is large, and consequently, the integrator has large errors. This analysis will be presented in a future publication.

\section{Conclusion}

In this article we have primarily studied HMC  in the exact integration scenario. The RHMC method introduced here is a version of HMC where the durations of the Hamiltonian flows are independent, exponentially distributed random variables.  This method has an infinitesimal generator which is a linear combination of a momentum randomization operator and a (differential) Liouville operator.

The analysis in Section~\ref{sec:non_asymptotic} related the non-asymptotic properties of RHMC to those of the underdamped Langevin dynamics.  We showed that, under standard  hypotheses on the potential energy function, RHMC possesses a stochastic Lyapunov function of the same form as that of the underdamped Langevin dynamics. However, unlike underdamped Langevin dynamics, the trajectories produced by RHMC are not continuous functions of time due to the instantaneous momentum randomizations.  This difference in regularity made it tricky to use standard approximate controllability arguments to establish a minorization condition for RHMC and an alternative approach was required.   Our analysis showed that the mechanism for dissipation (and hence, exponential stability) in RHMC comes from the momentum randomizations; therefore momentum randomizations play here the role played by the heat bath appearing in underdamped Langevin dynamics.  For sampling, the main qualitative advantage of RHMC compared with underdamped Langevin dynamics is that the paths of RHMC have a stronger tendency to move consistently away from the current state of the chain.  

In a model test problem, we  carried out a quantitative analysis of the sampling performance of RHMC using integrated autocorrelation time and mean squared displacement as metrics.  Our analysis  showed that randomizing the durations of the Hamiltonian flows mitigates some artifacts associated to using Hamiltonian dynamics. In particular, we saw that these sampling metrics depend monotonically on the mean duration parameter.  Numerical examples showed that this monotonicity persists for more general target distributions.   In contrast, these sampling metrics for classical HMC are a more complicated function of its (deterministic) duration parameter.

As an outlook to future developments of RHMC, we considered two approximations of RHMC based on spatially discretizing the infinitesimal
generator of RHMC.  This outlook introduced a new viewpoint to developing Hamiltonian-based MCMC methods.  In particular, it transforms
the problem of approximating the Hamiltonian flow from one of time discretizing Hamilton's equations into one of spatially discretizing its associated Liouville operator in high dimension in a way that generates a Markov process \cite{BoVa2016}.  We showed how to construct such approximations so that they preserve the Boltzmann-Gibbs distribution.   A complete analysis of these approximations to RHMC will be the subject of future work.

\appendix

\section{Harris Theorem}  \label{sec:harris_theorem}

%
%

Harris Theorem states that if a Markov process admits a Lyapunov function such that its sublevel sets are `small', then it is geometrically ergodic.  In this part, we recall this theorem for the convenience of the reader.  Since RHMC has an infinitesimal generator, it is convenient to formulate the Foster-Lyapunov condition in Harris Theorem in terms of an infinitesimal generator.   To make the domain of this generator sufficiently inclusive, we define the infinitesimal generator of a Markov process $\mathsf{X}(t)$ on a Polish state space $\Omega$ equipped with probability measure $\mathbb{P}$ in the following way.

\begin{defn} \label{defn:generator}
Let $D(\mathsf{L})$ be the set of all measurable functions $F: \Omega \to \mathbb{R}$ such that there exists a measurable function $G: \Omega \to \mathbb{R}$ with the property that for any $x \in \Omega$ and for all $t>0$ the process \[
  F(\mathsf{X}(t)) - F(x) -  \int_0^t G(\mathsf{X}(s)) ds \;, \quad \mathsf{X}(0)=x \;,
\]
is a local martingale adapted to the natural filtration of $\mathsf{X}$ under the probability measure $\mathbb{P}$.   Then we define $\mathsf{L} F= G$ and call $\mathsf{L}$ the infinitesimal generator of the process $\mathsf{X}(t)$ with domain $D(\mathsf{L})$.
\end{defn}

This definition seems to be due to Ref.~\cite[Definition (14.15) in Chapter 1]{Da1993}. 
In addition to the infinitesimal generator, denote by $\mathsf{P}_t$ the Markov semigroup of $\mathsf{X}(t)$, and denote the transition probabilities of $\mathsf{X}(t)$ by \[
\mathsf{\Pi}_{t,x}( A) = \Pr( \mathsf{X}(t) \in A \mid \mathsf{X}(0) = x ) \quad \forall t \ge 0 \;, \quad  \forall x \in \Omega \;.
\]
Sufficient conditions for Harris Theorem to hold are Hypotheses~\ref{driftcondition} and ~\ref{minorization} given below.

\begin{hypothesis}[Foster-Lyapunov Drift Condition] \label{driftcondition}
There exist a function $\Psi: \Omega \to \mathbb{R}^+$ and strictly positive constants $\mathsf{w}$ and $\mathsf{k}$ such that  \begin{equation} \label{foster_lyapunov_drift_condition}
\mathsf{L} \Psi ( x) \le - \mathsf{w} \; \Psi(x) + \mathsf{k}  \;,
\end{equation}
for all $x \in \Omega$.  
\end{hypothesis}

\begin{rem} \label{rem:driftcondition}
Hypothesis~\ref{driftcondition} implies that \[
\mathsf{P}_t \Psi(x) \le e^{-\mathsf{w} t} \Psi(x) + \frac{\mathsf{k}}{\mathsf{w}} (1-e^{-\mathsf{w} t} )
\]
for every $t\ge0$ and for every $x \in \Omega$.
\end{rem}

\begin{hypothesis}  \label{minorization}
There exist $\mathsf{a} > 0$ and $t>0$ such that the sublevel set $\{ z \in \Omega \mid  \Psi(z) \le 2 \mathsf{k}/\mathsf{w}\}$
is `small' i.e.
\begin{equation} \label{weakerminor}
\| \mathsf{\Pi}_{t,x}- \mathsf{\Pi}_{t,y} \|_\TV \le 2(1 - \mathsf{a})
\end{equation}
for every pair $x, y$ satisfying $\Psi(x) \vee \Psi(y) \le 2 \mathsf{k}/\mathsf{w}$, where the constants $\mathsf{k}$ and $\mathsf{w}$ are taken from Hypothesis~\ref{driftcondition}.
\end{hypothesis}

%
%

\begin{theorem}[Harris Theorem] \label{uncountable_harris}
Consider a Markov process $\mathsf{X}(t)$ on $\Omega$ with generator $\mathsf{L}$ and transition probabilities $\mathsf{\Pi}_{t,x}$,  which satisfies  Hypotheses~\ref{driftcondition}  and~\ref{minorization}.   Then $\mathsf{X}(t)$ possesses a unique invariant probability measure $\mathsf{\Pi}$, and there exist positive constants $\mathsf{C}$ and $\mathsf{r}$ (both depending only on the constants $\mathsf{w}$, $\mathsf{k}$ and $\mathsf{a}$ appearing in the assumptions) such that
\[
\| \mathsf{\Pi}_{t,x} - \mathsf{\Pi} \|_\TV  \le \mathsf{C} \; \exp(-\mathsf{r} \; t)  \; \Psi(x)\;,
\]
for all $x \in \Omega$ and for any $t \ge 0$.
\end{theorem}

\begin{rem}
It follows from \eqref{foster_lyapunov_drift_condition} that: \[
\mathsf{\Pi}( \Psi ) \le \frac{ \mathsf{k} }{ \mathsf{w} } \;.
\]
\end{rem}

For further exposition and a proof of Harris Theorem in a general context, see the monograph \cite{MeTw2009} (or Ref.~\cite{HaMa2011} for an alternative proof), and for a treatment in the specific context of stochastic differential equations see Ref.~\cite{MaStHi2002}.

\bigskip

\section*{Acknowledgements}
We wish to thank Christof Andrieu, Paul Dupuis and Andreas Eberle for useful discussions.  We also wish to thank the anonymous referee and associate editor for their careful reading of the mansucript.

\bibliographystyle{amsplain}
\bibliography{nawaf}

\end{document}